\documentclass[11pt]{article}
\usepackage{amsmath,amsthm,amsfonts,amssymb,enumerate}
\usepackage[colorlinks=true,citecolor=black,linkcolor=black,urlcolor=blue]{hyperref}
\usepackage{algorithm}
\usepackage{algorithmic} 
\usepackage{mdwlist}  
 \floatname{algorithm}{Algorithm}

\usepackage[lmargin=30mm,rmargin=30mm,bmargin=30mm,tmargin=30mm]{geometry}
\usepackage[numbers,sort&compress]{natbib}

\renewcommand{\baselinestretch}{1.1}
\allowdisplaybreaks

\theoremstyle{plain}
\newtheorem{theorem}{Theorem}
\newtheorem{lemma}[theorem]{Lemma}
\newtheorem{corollary}[theorem]{Corollary}

\newenvironment{definition}[1][Definition]{\begin{trivlist}
\item[\hskip \labelsep {\bfseries #1}]}{\end{trivlist}}

\newcommand{\ceil}[1]{\lceil{#1}\rceil}
\newcommand{\floor}[1]{\lfloor{#1}\rfloor}

\renewcommand{\geq}{\geqslant}
\renewcommand{\leq}{\leqslant}

\DeclareMathOperator{\tw}{tw}
\DeclareMathOperator{\pw}{pw}
\DeclareMathOperator{\bn}{bn}

\begin{document}

\title{\bf Treewidth of the Line Graph of Complete and Complete Multipartite Graphs \footnotemark[1]}

\author{Daniel J.\ Harvey \footnotemark[3]\qquad  David~R.~Wood\,\footnotemark[4]}

\date{\today}

\maketitle

\begin{abstract}
In recent papers by Grohe and Marx, the treewidth of the line graph of the complete graph is a critical example. We determine the exact treewidth of the line graph of the complete graph. By extending these techniques, we determine the exact treewidth of the line graph of a regular complete multipartite graph. For an arbitrary complete multipartite graph, we determine the treewidth of the line graph up to a lower order term.
\end{abstract}

\footnotetext[1]{This paper was previously titled ``Treewidth of Line Graphs".}
\footnotetext[3]{Department of Mathematics and Statistics, The
  University of Melbourne, Melbourne, Australia
  (\texttt{d.harvey@pgrad.unimelb.edu.au}). Research of D.J.H. is supported by an Australian Postgraduate Award.}
\footnotetext[4]{School of Mathematical Sciences, Monash University, Melbourne, Australia (\texttt{woodd@unimelb.edu.au}) Research of D.W. is supported by the Australian Research Council.}

\section{Introduction}
\label{section:intro}
The \emph{treewidth} $\tw(G)$ of a graph $G$ is a graph invariant used to measure how ``tree-like" $G$ is. It is of particular importance in structural and algorithmic graph theory; see the surveys \cite{bod, reed}. $\tw(G)$ is the minimum width of a \emph{tree decomposition} of $G$, which is defined as follows:

\begin{definition}
A \emph{tree decomposition} of a graph $G$ is a pair $(T, \{A_{x} \subseteq V(G) : x \in V(T)\})$ such that:
\begin{itemize*}
\item T is a tree.
\item $\{A_{x} \subseteq V(G) : x \in V(T)\}$ is a collection of sets of vertices of $G$, each called a \emph{bag}, indexed by the nodes of $T$.
\item For all $v \in V(G)$, the nodes of $T$ indexing the bags containing $v$ induce a non-empty (connected) subtree of $T$.
\item For all $vw \in E(G)$, there exists a bag of $T$ containing both $v$ and $w$.
\end{itemize*}
\end{definition}

The \emph{width} of a tree decomposition is the maximum size of a bag of $T$, minus 1. This minus 1 is added to ensure that every tree has treewidth 1.
Similarly, we can define pathwidth $pw(G)$ to be the minimum width of a tree decomposition where the underlying tree is a path.

The line-graph $L(G)$ of a graph $G$ is the graph with $V(L(G)) = E(G)$, such that two vertices of $L(G)$ are adjacent when the corresponding edges of $G$ are incident at a vertex.

In recent papers by Marx \cite{marx} and Grohe and Marx \cite{marxgrohe}, the treewidth of the line graph of the complete graph is a critical example. For a graph $G$, let $G^{(q)}$ denote the graph created by replacing each vertex of $G$ with a clique of size $q$ and replacing each edge between two vertices with all of the edges between the new two cliques. Marx \cite{marx} shows that if $G$ has large treewidth, then $G^{(q_1)}$ contains the $L(K_n)^{(q_2)}$ as a minor (for appropriate choices of $q_1$ and $q_2$). Grohe and Marx \cite{marxgrohe} show that $\tw(L(K_{n})) \geq \frac{\sqrt{2}-1}{4}n^{2} + O(n)$. In this paper, we determine $\tw(L(K_{n}))$ exactly. In doing so, our minimum width tree decomposition will be a path, so this result also holds for pathwidth.

\begin{theorem}
\label{theorem:KN}
$$
\tw(L(K_n)) = \pw(L(K_{n})) =
\begin{cases}
(\frac{n-1}{2})(\frac{n-1}{2}) + n-2 & \text{, if $n$ is odd} \\
(\frac{n-2}{2})(\frac{n}{2}) + n-2 & \text{, if $n$ is even}
\end{cases}
$$
\end{theorem}

The complete multipartite graph $K_{n_{1}, n_{2}, \dots, n_{k}}$ is the graph with $k$ colour classes, of order $n_{1}, \dots, n_{k}$ respectively, containing an edge between every pair of differently coloured vertices. We determine bounds on the treewidth of the line graph of the complete multipartite graph. Again, this is also a bound on the pathwidth.

\begin{theorem}
\label{theorem:KNk}
If $k \geq 2$ and $n = |V(K_{n_{1}, \dots, n_{k}})|$, then
\begin{align*}
-n(k-1) + \frac{3}{4}k(k-1) -1 + &\frac{1}{2}\left(\sum\limits_{1 \leq i < j \leq k} n_{i}n_{j}\right) \\ \leq \tw(K_{n_{1}, \dots, n_{k}}) &\leq \pw(K_{n_{1}, \dots, n_{k}}) \\  &\leq \frac{1}{2}n(k+5) + \frac{1}{4}k(k-1) - 4 + \frac{1}{2}\left(\sum\limits_{1 \leq i < j \leq k} n_{i}n_{j}\right).
\end{align*}
\end{theorem}

Theorem~\ref{theorem:KNk} implies that when $n_{1} = \dots = n_{k} = c$, (that is, when our complete multipartite graph is regular) then $\tw(L(K_{c, \dots, c})) \approx \frac{k^{2}c^{2}}{4}$ (ignoring the lower order terms). We improve this result, obtaining an exact answer for the treewidth and pathwidth of the line graph of a regular complete multipartite graph.

\begin{theorem}
\label{theorem:KNkreg}
If $n_{1} = n_{2} = \dots = n_{k} = c \geq 1$, then
$$\tw(L(K_{n_{1}, \dots, n_{k}})) = \pw(L(K_{n_{1}, \dots, n_{k}}))=
\begin{cases}
\frac{c^{2}k^{2}}{4} - \frac{c^{2}k}{4} + \frac{ck}{2} - \frac{c}{2} + \frac{k}{4} - \frac{5}{4} &  \text{, if $k$ odd, $c$ odd} \\
\frac{c^{2}k^{2}}{4} - \frac{c^{2}k}{4} + \frac{ck}{2} - \frac{c}{2} - 1 &  \text{, if $c$ even} \\
\frac{c^{2}k^{2}}{4} - \frac{c^{2}k}{4} + \frac{ck}{2} - \frac{c}{2} + \frac{k}{4} - \frac{3}{2} &  \text{, if $k$ even, $c$ odd} \\
\end{cases}
$$
\end{theorem}

Note that this implies Theorem~\ref{theorem:KN} is a specical case of Theorem~\ref{theorem:KNkreg}. In order to prove these results, we use the theory of \emph{brambles} and the Treewidth Duality Theorem, which we present in Section~\ref{section:bramtddual}. Section~\ref{section:lbnltd} presents a framework for proving results about the treewidth of general line graphs, which are of independent interest. Theorem~\ref{theorem:KN} is proved in Section~\ref{section:KN}. Theorem~\ref{theorem:KNk} and Theorem~\ref{theorem:KNkreg} are proved in Section~\ref{section:KNk} and Section~\ref{section:KNknonregtree}. 

Finally, note the following conventions: if $S$ is a subgraph of a graph $G$ and $x \in V(G)-V(S)$, then let $S \cup \{x\}$ denote the subgraph of $G$ with vertex set $V(S) \cup \{x\}$ and edge set $E(S) \cup \{xy: y \in S, xy \in E(G)\}$. Similarly, if $u \in V(S)$, let $S - \{u\}$ denote the subgraph with vertex set $V(S) - \{u\}$ and edge set $E(S) - \{uw: w \in S-\{u\}\}$.

\section{Brambles and the Treewidth Duality Theorem}
\label{section:bramtddual}
A \emph{bramble} of a graph $G$ is a collection $\mathcal{B}$ of connected subgraphs of $G$ such that each pair of subgraphs $X,Y \in \mathcal{B}$ touch, where $X$ and $Y$ \emph{touch} when they either have at least one vertex in common, or there exists an edge in $G$ with one end in $V(X)$ and the other in $V(Y)$.
The \emph{order} of a bramble is the size of the smallest hitting set $H$, where a \emph{hitting set} of a bramble $\mathcal{B}$ is a set of vertices $H$ such that $H \cap V(X) \neq \emptyset$ for all $X \in \mathcal{B}$.
For a given graph $G$, the \emph{bramble number} $\bn(G)$ is the maximum order of a bramble of $G$.
Brambles are important due to the following theorem of Seymour and Thomas \cite{dual}:
\begin{theorem}(Treewidth Duality Theorem)
For every graph $G$, $\bn(G) = \tw(G)+1$.
\end{theorem}  

In this paper we employ the following standard approach for determining the treewidth and pathwidth of a particular graph $G$. First we construct a bramble of large order, thus proving a lower bound on $\tw(G)$. Then to prove an upper bound, we construct a path decomposition of small width. A first step in constructing such a path decomposition is to place a minimum hitting set of the bramble in a single bag; when this bag is a bag of maximum size, we have an exact answer for $\pw(G)$ and $\tw(G)$.

\section{Line-brambles}
\label{section:lbnltd}
Throughout this section, let $G$ be an arbitrary graph. In order to construct a bramble of the line graph $L(G)$, we define the following:

\begin{definition}
A \emph{line-bramble} $\mathcal{B}$ of $G$ is a collection of connected subgraphs of $G$ satisfying the following properties:
\begin{itemize}
\item For all $X \in \mathcal{B}$, $|V(X)| \geq 2$.
\item For all $X,Y \in \mathcal{B}$, $V(X) \cap V(Y) \neq \emptyset$.
\end{itemize}
Define a \emph{hitting set} for a line-bramble $\mathcal{B}$ to be a set of edges $H \subseteq E(G)$ that intersects each $X \in \mathcal{B}$. Then define the \emph{order} of $\mathcal{B}$ to be the size of the minimum hitting set $H$ of $\mathcal{B}$.
\end{definition}

\begin{lemma}
\label{lemma:KNlbisb}
Given a line-bramble $\mathcal{B}$ of $G$, there is a bramble $\mathcal{B'}$ of $L(G)$ of the same order.
\end{lemma}
\begin{proof}
Let $X$ be an element of line-bramble $\mathcal{B}$ and let $\mathcal{B'} = \{E(X)| X \in \mathcal{B}\}$ (where here $E(X)$ refers to the subgraph of $L(G)$ induced by the vertices of $E(X)$). Recall $X$ is connected. Now since $|V(X)| \geq 2$, $X$ contains an edge. So $E(X)$ induces a non-empty connected subgraph of $L(G)$. Consider $E(X)$ and $E(Y)$ in $\mathcal{B'}$. Thus $V(X) \cap V(Y) \neq \emptyset$. Let $v$ be a vertex in $V(X) \cap V(Y)$. Then there exists some $xv \in E(X)$ and $vy \in E(Y)$, and so in $L(G)$ there is an edge between the vertex $xv$ and the vertex $vy$. Hence $E(X)$ and $E(Y)$ touch. Thus $\mathcal{B'}$ is a bramble of $L(G)$.
All that remains is to ensure $\mathcal{B}$ and $\mathcal{B'}$ have the same order. If $H$ is a minimum hitting set for $\mathcal{B}$, then $H$ is also a set of vertices in $L(G)$ that intersects a vertex in each $E(X) \in \mathcal{B'}$. So $H$ is a hitting set for $\mathcal{B'}$ of the same size. Conversely, if $H'$ is a minimum hitting set of $\mathcal{B'}$, then $H'$ is a set of edges in $G$ that contains an edge in each $X \in \mathcal{B}$. So $H'$ is a hitting set for $\mathcal{B}$. Thus the orders of $\mathcal{B}$ and $\mathcal{B'}$ are equal.
\end{proof}

Hence, in order to determine a lower bound on the bramble number $\bn(L(G))$, it is sufficient to construct a line-bramble of $G$ of large order. We will now define a particular line-bramble for any graph $G$ with $|V(G)| \geq 3$.

\begin{definition}
Given a vertex $v \in V(G)$, the \emph{canonical line-bramble for $v$} of $G$ is the set of connected subgraphs $X$ of $G$ such that either $|V(X)| > \frac{|V(G)|}{2}$, or $|V(X)| = \frac{|V(G)|}{2}$ and $X$ contains $v$. Note that if $|V(G)|$ is odd, then no elements of the second type occur.
\end{definition}

\begin{lemma}
\label{lemma:KNislbram}
For every graph $G$ with $|V(G)| \geq 3$ and for all $v \in V(G)$, the canonical line-bramble for $v$, denoted $\mathcal{B}$, is a line-bramble of $G$.
\end{lemma}
\begin{proof}
By definition, each element of $\mathcal{B}$ is a connected subgraph. Since $|V(G)| \geq 3$, each element of $\mathcal{B}$ contains at least two vertices. All that remains to show is that each pair of subgraphs $X$,$Y$ in $\mathcal{B}$ intersect in at least one vertex. If $|V(X)|=|V(Y)| = \frac{|V(G)|}{2}$, then $X$ and $Y$ intersect at $v$. Otherwise, without loss of generality, $|V(X)| > \frac{|V(G)|}{2}$ and $|V(Y)| \geq \frac{|V(G)|}{2}$. If $V(X) \cap V(Y) = \emptyset$, then $|V(X) \cup V(Y)| = |V(X)|+|V(Y)| > |V(G)|$, which is a contradiction.  
\end{proof}

Let $v \in V(G)$ be an arbitrary vertex and let $H$ be a minimum hitting set of $\mathcal{B}$, the canonical line-bramble for $v$. Consider the graph $G-H$. $H$ is a set of edges, so $V(G-H)=V(G)$. Then each component of $G-H$ contains at most $\frac{|V(G)|}{2}$ vertices, otherwise some component of $G-H$ contains an element of $\mathcal{B}$ that does not contain an edge of $H$. Similarly, if a component contains $\frac{|V(G)|}{2}$ vertices, it cannot contain the vertex $v$.
Thus, our hitting set $H$ must be large enough to separate $G$ into such components. 
The next lemma follows directly:
\begin{lemma}
\label{lemma:KNlegithit}
For every graph $G$ with $|V(G)| \geq 3$ and for all $v \in V(G)$, a set $H \subseteq E(G)$ is a hitting set of $\mathcal{B}$, the canonical line-bramble for $v$, if and only if every component of $G-H$ has at most $\frac{|V(G)|}{2}$ vertices, and $v$ is not in a component that contains exactly $\frac{|V(G)|}{2}$ vertices.
\end{lemma}

Note the similarity between this characterisation and the \emph{bisection width} of a graph (see \cite{bisone,bistwo}, for example), which is the minimum number of edges between any $A,B \subset V(G)$ where $A \cap B = \emptyset$ and $|A|=\floor{\frac{|V(G)|}{2}}$ and $|B|=\ceil{\frac{|V(G)|}{2}}$. (Later we show that most of our components have maximum or almost maximum allowable order.) 

\begin{lemma}
\label{lemma:minHrule}
Let $G$ be a graph with $|V(G)| \geq 3$, $v$ a vertex of $G$. and $H$ a minimum hitting set for $\mathcal{B}$, the canonical line-bramble for $v$. Then no edge of $H$ has both endpoints in the same component of $G-H$.
\end{lemma}
\begin{proof}
For the sake of a contradiction assume that both endpoints of an edge $e \in H$ are in the same component of $G-H$. Then consider the set $H-e$. By Lemma~\ref{lemma:KNlegithit}, $H-e$ is a hitting set of $\mathcal{B}$, since the vertex sets of the components of $G-H$ have not changed. But $H-e$ is smaller than the minimum hitting set $H$, a contradiction.
\end{proof}

%
%
%

\section{Line Graph of the Complete Graph}
\label{section:KN}
We now prove Theorem~\ref{theorem:KN}. Let $G:=K_{n}$. When $n \leq 2$, Theorem~\ref{theorem:KN} holds trivially, so we can assume $n \geq 3$. Firstly, we shall determine a lower bound by considering the canonical line-bramble for $v$, denoted $\mathcal{B}$. Given that $K_n$ is regular, we shall choose vertex $v$ of $K_n$ arbitrarily. 

If $H$ is a minimum hitting set of $\mathcal{B}$, label the components of $G-H$ as $Q_{1}, \dots , Q_{p}$ such that $|V(Q_{1})| \geq |V(Q_{2})| \geq \dots \geq |V(Q_{p})|$. We refer to this as labelling the components \emph{descendingly}.

Consider a pair of components $(Q_{i},Q_{j})$ where $i<j$ and the components are labelled descendingly. We call this a \emph{good pair} if one of the following conditions hold:
\begin{enumerate*}
\item $|V(Q_{i})| < \frac{n}{2} - 1$,
\item $n$ is even, $|V(Q_{i})| = \frac{n}{2} - 1$, $V(Q_{j}) \neq \{v\}$, and $v \notin V(Q_{i})$.
\end{enumerate*}

\begin{lemma}
\label{lemma:KNnogoodpair}
Let $G$ be the complete graph with $n \geq 3$ vertices, $v$ a vertex of $G$, $\mathcal{B}$ be the canonical line-bramble for $v$ of $G$, and $H$ a minimum hitting set of $\mathcal{B}$. If $Q_{1}, \dots, Q_{p}$ are the components of $G-H$ labelled descendingly, then $Q_{1}, \dots, Q_{p}$ does not contain a good pair. 
\end{lemma}
\begin{proof}
Say $(Q_{i},Q_{j})$ is a good pair. Let $x$ be a vertex of $Q_{j}$, such that if $(Q_{i},Q_{j})$ is of the second type, then $x \neq v$. Let $H'$ be the set of edges obtained from $H$ by removing the edges from $x$ to $Q_{i}$ and adding the edges from $x$ to $Q_{j}$. Then the components for $G-H'$ are $Q_{1}, \dots, Q_{i-1}, Q_{i} \cup \{x\}, Q_{i+1}, \dots, Q_{j-1}, Q_{j}-\{x\}, Q_{j+1}, \dots Q_{p}$. By Lemma~\ref{lemma:KNlegithit}, to ensure $H'$ is a hitting set, we only need to ensure that $V(Q_{i}) \cup \{x\}$ is sufficiently small, since all other components are the same as in $H$, or smaller. If $(Q_{i},Q_{j})$ is of the first type, then $|V(Q_{i}) \cup \{x\}| = |V(Q_{i})|+1 < \frac{n}{2}$. If $(Q_{i},Q_{j})$ is of the second type, $|V(Q_{i}) \cup \{x\}| = \frac{n}{2}$, but it does not contain $v$. Thus, by Lemma~\ref{lemma:KNlegithit}, $H'$ is a hitting set.
However, $|H'| = |H| - |V(Q_{i})| + |V(Q_{j})|-1 \leq |H|-1$, which contradicts that $H$ is a minimum hitting set.
\end{proof}

\begin{lemma}
\label{lemma:KNthreecomp}
Let $G,v,\mathcal{B}$ and $H$ be as in Lemma~\ref{lemma:KNnogoodpair}. Then $G-H$ has exactly three components.
\end{lemma}
\begin{proof}
Recall by Lemma~\ref{lemma:KNlegithit}, we have an upper bound on the order of the components of $G-H$. 
Firstly, we show that $G-H$ has at least three components. If $G-H$ has only one component, clearly this component is too large. If $G-H$ has two components and $n$ is odd, then one of the components must have more than $\frac{n}{2}$ vertices. If $G-H$ has two components and $n$ is even, it is possible that both components have exactly $\frac{n}{2}$ vertices, however one of these components must contain $v$. Thus $G-H$ has at least three components. Now, assume $G-H$ has at least four components. We will show that it has a good pair, contradicting Lemma~\ref{lemma:KNnogoodpair}. Label the components of $G-H$ descendingly.

If $n$ is odd, we have a good pair of the first type when any two components have less than $\frac{n-1}{2}$ vertices. Thus at least three components have order at least $\frac{n-1}{2}$. Then $|V(G)| \geq 3(\frac{n-1}{2})+1>n$ when $n \geq 2$, which is a contradiction.

If $n$ is even, we have the first type of good pair whenever two components have less than $\frac{n}{2}-1$ vertices. Similarly to the previous case, $|V(G)| \geq 3(\frac{n}{2} - 1) +1 >n$, again a contradiction when $n>4$. If $n=4$ then each component is a single vertex. Take $Q_{i},Q_{j}$ to be two of these components, neither of which contain the vertex $v$. Then $(Q_{i},Q_{j})$ is a good pair of the second type.
Hence $G-H$ does not have more than three components, and as such it has exactly three components.
\end{proof}

\begin{lemma}
\label{lemma:KNcompbd}
Let $G,v,\mathcal{B}$ and $H$ be as in Lemma~\ref{lemma:KNnogoodpair}, and the components of $G-H$ be labelled descendingly.
If $n$ is odd then $|V(Q_{1})|=|V(Q_{2})|=\frac{n-1}{2}$ and $|V(Q_{3})|=1$. 
If $n$ is even then $|V(Q_{1})|= \frac{n}{2}, |V(Q_{2})|=\frac{n}{2}-1$ and $|V(Q_{3})|=1$.
\end{lemma}
\begin{proof}
Lemma~\ref{lemma:KNthreecomp} shows that $G-H$ has exactly three components.
By Lemma~\ref{lemma:KNnogoodpair}, $(Q_{2},Q_{3})$ is not a good pair. Hence $|V(Q_{1})| \geq |V(Q_{2})| \geq \frac{n-1}{2}$ when $n$ is odd, and $|V(Q_{1})| \geq |V(Q_{2})| \geq \frac{n}{2}-1$ when $n$ is even, or else we have a good pair of the first type. By Lemma~\ref{lemma:KNlegithit}, when $n$ is odd, $|V(Q_{1})| = |V(Q_{2})| = \frac{n-1}{2}$ and so $|V(Q_{3})|=1$.
When $n$ is even, however, $\frac{n}{2}-1 \leq |V(Q_{1})|,|V(Q_{2})| \leq \frac{n}{2}$. Since $Q_{3}$ is not empty, it follows that $|V(Q_{3})|=1$ or $2$. If $|V(Q_{3}|=1$, then $|V(Q_{1})|= \frac{n}{2}, |V(Q_{2})|=\frac{n}{2}-1$ and $|V(Q_{3})|=1$, as required. Otherwise, $|V(Q_{1})|,|V(Q_{2})|=\frac{n}{2}-1$. But then at least one of $Q_{1},Q_{2}$ does not contain $v$, and $V(Q_{3}) \neq \{v\}$. Thus either $(Q_{1},Q_{3})$ or $(Q_{2},Q_{3})$ is a good pair of the second type, contradicting Lemma~\ref{lemma:KNnogoodpair}. 
\end{proof}

\begin{lemma}
\label{lemma:KNHorder}
Let $G,v,\mathcal{B}$ and $H$ be as in Lemma~\ref{lemma:KNnogoodpair}. Then $|H| \geq (\frac{n-1}{2})(\frac{n-1}{2}) + (n-1)$ when $n$ is odd, and $|H| \geq (\frac{n-2}{2})(\frac{n}{2}) + (n-1)$ when $n$ is even.
\end{lemma}
\begin{proof}
From Lemma~\ref{lemma:KNcompbd} we know the order of the components of $G-H$. $H$ contains at least every edge between each pair of components, and since $G$ is complete there is an edge for each pair of vertices. From this it is easy to calculate $|H|$.
\end{proof}

Lemma~\ref{lemma:KNHorder} and the Treewidth Duality Theorem imply:

\begin{corollary}
\label{corollary:KNbrno}
Let $G$ be the complete graph with $n \geq 3$ vertices. $$\tw(L(K_{n})) = \bn(L(K_{n})) -1 \geq 
\begin{cases}
(\frac{n-1}{2})(\frac{n-1}{2}) + (n-2) & \text{, if $n$ is odd} \\
(\frac{n-2}{2})(\frac{n}{2}) + (n-2) & \text{, if $n$ is even.}
\end{cases}
$$
\end{corollary}

Now, to obtain an upper bound on $\pw(L(G))$, we construct a path decomposition of $L(G)$. First, label the vertices of $G$ by $1, \dots, n$. Let $T$ be an $n$-node path, also labelled by $1, \dots, n$. The bag $A_{i}$ for the node labelled $i$, is defined such that $A_{i} = \{ij : j \in V(G)\} \cup \{uw : u<i<w\}$. For a given $A_i$, call the edges in $\{ij : j \in V(G)\}$ \emph{initial edges} and call the edges in $\{uw : u<i<w\}$ \emph{crossover edges}. (Note here these edges of $G$ are really acting as vertices of $L(G)$, but we refer to them as edges for simplicity.)

\begin{lemma}
\label{lemma:KNtd}
Let $G$ be the complete graph with $n \geq 3$ vertices. $(T,\{A_{1}, \dots, A_{n}\})$ is a path decomposition for $L(G)$ of width

$$\begin{cases}
(\frac{n-1}{2})(\frac{n-1}{2}) + (n-2) & \text{, if $n$ is odd} \\
(\frac{n-2}{2})(\frac{n}{2}) + (n-2) & \text{, if $n$ is even.}
\end{cases}$$
\end{lemma}
\begin{proof}
Each edge $uw$ of $G$ appears in $A_{u}$ and $A_{w}$ as initial edges. Similarly, all of the edges incident at the vertex $u$ appear in $A_{u}$, and the same holds for $w$. Observe that $uw$ is in $A_{i}$ if and only if $u \leq i \leq w$. Thus the nodes indexing the bags containing $uw$ form a connected subtree of $T$, as required.

Now we determine the size of $A_{i}$. $A_{i}$ contains $n-1$ initial edges and $(i-1)(n-i)$ crossover edges. So $|A_{i}|=(n-1) + (i-1)(n-i)$. This is maximised when $i=\frac{n+1}{2}$ if $n$ is odd, and when $i=\frac{n}{2}$ or $\frac{n+2}{2}$ if $n$ is even. From this we can calculate the largest bag size, and hence the width of $T$.
\end{proof}

Lemma~\ref{lemma:KNtd} gives an upper bound on $\pw(L(K_n))$ and $\tw(L(K_n))$. This, combined with the lower bound in Corollary~\ref{corollary:KNbrno}, completes the proof of Theorem~\ref{theorem:KN}.

%
%

\section{Line-brambles of the Complete Multipartite Graph}
\label{section:KNk}
We now extend the above result to the line graph $L(G)$ of the complete multipartite graph $G:=K_{n_{1}, \dots, n_{k}}$, where $k \geq 2$. Let $n:=|V(G)|=n_{1}+\dots +n_{k}$. If $n=k$, then $G=K_{n}$ and Theorem~\ref{theorem:KN} determines $\tw(G)$ exactly, so we may assume $n>k$. Let $X_{i}$ be the $i^{th}$ colour class of $G$, with order $n_{i}$. Call $X_{i}$ \emph{odd} or \emph{even} depending on the parity of $|X_{i}|$. In a similar fashion to Section~\ref{section:KN} we shall first find a line-bramble of $G$.

As we did previously, we shall consider a canonical line-bramble for $v$ denoted $\mathcal{B}$. However, we shall choose vertex $v$ from a colour class of largest order. Note that $v$ has minimum degree. Let $H$ be a hitting set of $\mathcal{B}$, and label the components of $G-H$ by $Q_{1}, \dots, Q_{p}$. Denote $H$ and the labelling of its components together as $(H, (Q_1, \dots, Q_p))$. Choose $(H, (Q_1, \dots, Q_p))$ such that the following conditions hold, in order of preference:
\begin{description*}
\item[(0)] $|H|$ is minimised,
\item[(1)] $|V(Q_{1})|$ is maximised,
\item[(2)] $|V(Q_{2})|$ is maximised,
\item[] \indent $\vdots$ 
\item[(p)] $|V(Q_{p})|$ is maximised,
\item[(p+1)] $v$ is in the component of highest possible index.
\end{description*}  
By condition \textbf{(0)}, $H$ is a minimum hitting set. Note, as a result of this that $|V(Q_{1})| \geq |V(Q_{2})| \geq \dots \geq |V(Q_{p})|$, otherwise we can keep $H$ and easily find a better choice of labelling. Call a choice of $(H, (Q_1, \dots, Q_p))$ that satisifies these conditions a \emph{good labelling}.

Consider a pair of components $(Q_{i},Q_{j})$ where $i<j$ and $Q_1, \dots, Q_p$ is from a good labelling. We call this a \emph{good pair} when for all $x \in Q_{j}$ there exists $y \in Q_{i}$ such that $xy$ is an edge, and one of the following holds:
\begin{enumerate*}
\item $|V(Q_{i})| < \frac{n}{2} - 1$,
\item $n$ is even, $|V(Q_{i})| = \frac{n}{2} - 1$, $v \notin V(Q_{i})$ and $V(Q_{j}) \cap X_{s} \neq \{v\}$ for all colour classes $X_{s}$.
\end{enumerate*}

\begin{lemma}
\label{lemma:KNknogoodpair}
Let $G$ be a complete multipartite graph $G:=K_{n_{1}, \dots, n_{k}}$ such that $k \geq 2$ and $n > k$, $v$ a vertex of $G$ chosen from a largest colour class, $\mathcal{B}$ a canonical line-bramble for $v$, and $(H, (Q_1, \dots, Q_p))$ a good labelling. Then $Q_1, \dots, Q_p$ does not contain a good pair.
\end{lemma}
\begin{proof}
Assume $(Q_{i},Q_{j})$ is a good pair. For each $X_{s}$ that intersects $Q_{j}$, let $x_{s}$ be some vertex of $Q_{j} \cap X_{s}$. If $(Q_{i},Q_{j})$ is of the second type, choose each $x_{s} \neq v$. Let $H_{s}$ be the set of edges created by taking $H$ and removing the edges from $x_{s}$ to $Q_{i}$, then adding the edges from $x_{s}$ to $(Q_{j} - X_s)$. Thus we have removed $|V(Q_{i})|-|V(Q_{i}) \cap X_{s}|$ edges and have added $|V(Q_{j})|-|V(Q_{j}) \cap X_{s}|$. 

Suppose that $|V(Q_{j})|-|V(Q_{j}) \cap X_{s}| > |V(Q_{i})|-|V(Q_{i}) \cap X_{s}|$ for each $X_{s}$ that intersects $Q_{j}$. Then $$\sum_{s:X_{s} \cap V(Q_{j}) \neq \emptyset} |V(Q_{j})|-|V(Q_{j}) \cap X_{s}| > \sum_{s:X_{s} \cap V(Q_{j}) \neq \emptyset} |V(Q_{i})|-|V(Q_{i}) \cap X_{s}|.$$ However, since we are cycling through all colour classes that intersect $Q_{j}$, $$\sum_{s:X_{s} \cap V(Q_{j}) \neq \emptyset} |V(Q_{j}) \cap X_{s}| = |V(Q_{j})|.$$ If there are $r$ such colour classes, then $$(r-1)|V(Q_{j})| > r|V(Q_{i})|-\sum_{s:X_{s} \cap V(Q_{j}) \neq \emptyset}|V(Q_{i}) \cap X_{s}| \geq (r-1)|V(Q_{i})|.$$
This implies $|V(Q_{j})| > |V(Q_{i})|$, which is a contradiction of condition \textbf{(i)}. Hence, for some $s$, $|V(Q_{j})|-|V(Q_{j}) \cap X_{s}| \leq |V(Q_{i})|-|V(Q_{i}) \cap X_{s}|$. Fix such an $s$.

A component of $G-H_{s}$ is either one of $Q_{1}, \dots, Q_{i-1}, Q_{i+1}, \dots, Q_{j-1}, Q_{j+1}, \dots, Q_{p}$, or $Q_{i} \cup \{x_{s}\}$ (which is connected as $x_{s}$ has a neighbour in $Q_{i}$), or strictly contained within $Q_{j}$. Since $H$ is a hitting set, to prove $H_{s}$ is a hitting set it suffices to show that $Q_{i} \cup \{x_{s}\}$ is sufficiently small, by Lemma~\ref{lemma:KNlegithit}. If $(Q_{i},Q_{j})$ is of the first type, then $|V(Q_{i}) \cup \{x_{s}\}| = |V(Q_{i})|+1 < \frac{n}{2}$. So $V(Q_{i}) \cup \{x_{s}\}$ is sufficiently small. If $(Q_{i},Q_{j})$ is of the second type, $|V(Q_{i}) \cup \{x_{s}\}| = \frac{n}{2}$, but it does not contain $v$. Thus $H_{s}$ is a hitting set.
However, $|H_{s}| = |H| - (|V(Q_{i})|-|V(Q_{i}) \cap X_{s}|) + (|V(Q_{j})|-|V(Q_{j}) \cap X_{s}|) \leq |H|$. If $|H_{s}|<|H|$, then condition \textbf{(0)} is contradicted. If $|H_{s}|=|H|$, since $|V(Q_{i}) \cup \{x_{s}\}| > |V(Q_{i})|$ and only components of higher index have become smaller, $H_{s}$ is a better choice of minimum hitting set by condition \textbf{(i)}, which is a contradiction.
\end{proof}

\begin{lemma}
\label{lemma:KNk3pluscomp}
Let $G,v,\mathcal{B}$ and $(H, (Q_1, \dots, Q_p))$ be as in Lemma~\ref{lemma:KNknogoodpair}. $G-H$ has at least three components.
\end{lemma}
\begin{proof}
By Lemma~\ref{lemma:KNlegithit}, we have an upper bound on the order of the components of $G-H$. If $G-H$ has only one component, clearly this component is too large. If $G-H$ has two components and $n$ is odd, then one of the components must have more than $\frac{n}{2}$ vertices. If $G-H$ has two components and $n$ is even, it is possible that both components have exactly $\frac{n}{2}$ vertices, however one of these components must contain $v$. Thus $G-H$ has at least three components.
\end{proof}

If $G$ is a star $K_{1,n-1}$, then $L(G) \cong K_{n-1}$ and $\tw(L(G)) = n-2$, which satisfies Theorem~\ref{theorem:KNk}. Now assume that $G$ is not a star.
If $(H,(Q_{1}, \dots, Q_{p}))$ is a good labelling where $p \geq 4$ and $Q_{2}, \dots, Q_{p}$ are all singleton sets and contained within one colour class, then say that $(H,(Q_{1}, \dots, Q_{p}))$ is a \emph{rare configuration}.

\begin{lemma}
\label{lemma:KNknorareconfig}
Let $G$ be a complete multipartite graph $G:=K_{n_{1}, \dots, n_{k}}$ such that $k \geq 2, n>k$ and $G$ is not a star, $v$ a vertex of $G$ chosen from a largest colour class, $\mathcal{B}$ a canonical line-bramble for $v$, and $(H, (Q_1, \dots, Q_p))$ a good labelling. Then $(H,(Q_{1}, \dots, Q_{p}))$ is not a rare configuration.
\end{lemma}
\begin{proof}
Assume $G$ is a rare configuration, but $G$ is not a star. Let $X_{s}$ be the colour class of $Q_{2}, \dots, Q_{p}$. Since $p \geq 4$, we may choose $j \in \{2, \dots, p\}$ such that $V(Q_{j}) \neq \{v\}$.

Suppose that one of the following conditions hold:
\begin{itemize*}
\item $|V(Q_{1})| < \frac{n}{2}-1$, 
\item $n$ is even, $|V(Q_{1})| = \frac{n}{2}-1$ and $v \notin V(Q_{1})$.
\end{itemize*}
$Q_{1}$ must contain at least two vertices not in $X_{s}$ since $G$ is not a star or an independent set (as $k \geq 2$).
So for each $x \in V(Q_{2}) \cup \dots \cup V(Q_{p})$, there is some $y \in V(Q_{1})$ such that $y \notin X_{s}$, so the edge $xy$ exists. Then $(Q_{1},Q_{j})$ is a good pair, which contradicts Lemma~\ref{lemma:KNknogoodpair}. Thus by Lemma~\ref{lemma:KNlegithit}, 
$$|Q_{1}| = 
\begin{cases}
\frac{n-1}{2} & \text{, if $n$ is odd} \\
\frac{n}{2}-1 & \text{, if $n$ is even and $v \in V(Q_{1})$} \\
\frac{n}{2} & \text{, if $n$ is even and $v \notin V(Q_{1})$}
\end{cases}
$$
Since at least two vertices of $Q_{1}$ are not in $X_{s}$, we may choose $y \in (V(Q_{1})-\{v\})-X_{s}$. Say $y \in X_{t}$. We can assume that $v \in V(Q_{1})$ or $v \in V(Q_{p})$, since if $v \in V(Q_{2}) \cup \dots V(Q_{p-1})$, then we can relabel the components $Q_{2}, \dots, Q_{p}$ to obtain a choice of $(H,(Q_{1}, \dots, Q_{p})$ which is better with regards to the condition \textbf{(p+1)}. Thus let $z \in V(Q_{2})$, and so $z \neq v$ since $p \geq 4$.
Let $H'$ be the set of edges created by taking $H$ and removing the edges from $y$ to $Q_{3} \cup \dots \cup Q_{p-1}$, adding the edges from $y$ to $Q_{1}-X_t$, and removing the edges from $z$ to $Q_{1}-\{y\}$. Then the components of $G-H'$ are $Q_{1} \cup \{z\} -\{y\}$, $\{y\} \cup Q_{3} \cup \dots \cup Q_{p-1}$, $Q_{p}$. The component $Q_{1} \cup \{z\} - \{y\}$ is connected since $Q_{1} - \{y\}$ contains a vertex not in $X_{s}$ and $z \in X_{s}$. Similarly, $\{y\} \cup Q_{3} \cup \dots \cup Q_{p-1}$ is connected since $y \in X_{t}$ and all vertices of $Q_{3} \cup \dots \cup Q_{p-1}$ are in $X_{s}$. 

By Lemma~\ref{lemma:KNlegithit}, to show $H'$ is a hitting set, it is sufficient to show that no component of $G-H'$ is too large. Since $|V(Q_{1} \cup \{z\} - \{y\})|=|V(Q_{1})|$ and $v \neq z$ and $H$ is a hitting set, $Q_{1} \cup \{z\} -\{y\}$ is sufficiently small. Similarly $Q_{p}$ is sufficiently small. However, $|V(\{y\} \cup Q_{3} \cup \dots \cup Q_{p-1})|=p-2$. As $|V(Q_{1})|+ \dots + |V(Q_{p})|=n$, it follows that $p-2 = n - |V(Q_{1})|-1$. In order to show this is sufficiently small, we need to consider the parity of $n$, which we consider below. Also note, $$|H'| = |H| - (p-3) + (|V(Q_{1})|-|V(Q_{1}) \cap X_{t}|) - (|V(Q_{1})|-1 -|V(Q_{1}) \cap X_{s}|).$$ Since $|V(Q_{1}) \cap X_{t}| \geq 1$ and $|V(Q_{1}) \cap X_{s}| \leq |V(Q_{1})|-2$, we have $|H'| \leq |H| - (p-1)+|V(Q_{1})|=|H|+2|V(Q_{1})|-n$. This also depends on the parity of $n$. Now we consider these separate cases to check the order of $\{y\} \cup Q_{3} \cup \dots \cup Q_{p-1}$ and $|H'|$.

Firstly, say $n$ is odd. In this case $|V(Q_{1})|=\frac{n-1}{2}$, so then $|V(\{y\} \cup Q_{3} \cup \dots \cup Q_{p-1})| = p-2 = n - \frac{n-1}{2}-1 = \frac{n-1}{2}$, and so $\{y\} \cup Q_{3} \cup \dots \cup Q_{p-1}$ is sufficiently small, and $H'$ is a hitting set. Also, $|H'| \leq |H|+2(\frac{n-1}{2})-n < |H|$, which contradicts condition \textbf{(0)}.
Secondly, say $n$ is even and $v \in V(Q_{1})$. Then $|V(Q_{1})|=\frac{n}{2}-1$, implying $p-2 = \frac{n}{2}$, and $|H'| \leq |H|-2$. This contradicts condition \textbf{(0)}.
Finally, say $n$ is even and $v \notin V(Q_{1})$. Then $|V(Q_{1})|=\frac{n}{2}$ and $v \in V(Q_{p})$. Then $p-2 = \frac{n}{2}-1$, and $|H'| \leq |H|+2(\frac{n}{2})-n = |H|$. However, note that the order of the second largest component of $G-H'$ is $p-2=\frac{n}{2} -1$, whereas for $G-H$ the order of the second largest component is $1$. As $G$ is a rare configuration but not a star, $n \geq 5$, since $|V(Q_{1})| \geq 2$ and $p \geq 4$, implying $\frac{n}{2} -1 > 1$. Thus $H'$ is a better choice of minimum hitting set, by condition \textbf{(2)}.

Thus, in either case, if $G$ is not a star, but is a rare configuration, then there is a contradiction to one of our conditions on $(H,(Q_{1}, \dots, Q_{p}))$.
\end{proof}

\begin{lemma}
\label{lemma:KNkthreecomp}
Let $G,v,\mathcal{B}$ and $(H,(Q_1, \dots, Q_p))$ be as in Lemma~\ref{lemma:KNknorareconfig}. Then $G-H$ has exactly three components.
\end{lemma}
\begin{proof}
$G-H$ has at least three components, by Lemma~\ref{lemma:KNk3pluscomp}.
Assume for the sake of a contradiction that $G-H$ has greater than three components. Since $p \geq 4$, if all components but $Q_{1}$ are singleton sets in the one colour class, then we have a rare configuration. By Lemma~\ref{lemma:KNknorareconfig}, this cannot occur. Thus either $Q_{2}$ is not a singleton set, or $Q_{2}, \dots, Q_{p}$ are not all in one colour class. Consider a pair $(Q_{i},Q_{j})$, where $i \in \{1,2\}$, $i<j$ and if $|V(Q_{i})| = 1$ then $Q_{j}$ and $Q_{i}$ are not in the same colour class.
We can find such a pair for $i=1$ and for $i=2$ since this is not a rare configuration. In either case, for all $x \in V(Q_{j})$ there exists a $y \in V(Q_{i})$ such that $xy$ is an edge, since there is always some $y \in V(Q_{i})$ of a different colour class to $x$. Since $(Q_{i},Q_{j})$ is not a good pair by Lemma~\ref{lemma:KNknogoodpair}, we know $|V(Q_{i})|$ is too large. In particular, if $n$ is odd, $|V(Q_{1})|=|V(Q_{2})|=\frac{n-1}{2}$. However, since each component must contain a vertex and $p \geq 4$, the sum of the orders of the components is at least $2(\frac{n-1}{2})+2>n$, which is a contradiction. If $n$ is even and $v$ is in neither $Q_{1}$ nor $Q_{2}$, then $|V(Q_{1})|=|V(Q_{2})|=\frac{n}{2}$, which again means the sum of the orders of the components is too large. Finally, if $n$ is even and without loss of generality $v \in V(Q_{2})$, then $|V(Q_{1})|=\frac{n}{2}$ and $|V(Q_{2})|=\frac{n}{2}-1$, which still gives a contradiction on the orders of the components. Hence $G-H$ has exactly three components.
\end{proof}

\begin{lemma}
\label{lemma:KNkcompbd}
Let $G,v,\mathcal{B}$ and $(H,(Q_1, \dots, Q_p))$ be as in Lemma~\ref{lemma:KNknorareconfig}.
If $n$ is odd, then $|V(Q_{1})|= |V(Q_{2})|=\frac{n-1}{2}$ and $|V(Q_{3})|=1$. 
If $n$ is even, then $|V(Q_{1})|= \frac{n}{2}, |V(Q_{2})|=\frac{n}{2}-1$ and $|V(Q_{3})|=1$.
\end{lemma}
\begin{proof}
Lemma~\ref{lemma:KNkthreecomp} shows that $G-H$ has three components. Recall that in a good labelling that $|V(Q_{1})| \geq |V(Q_{2})| \geq |V(Q_{3})|$. 
If $|V(Q_{1})|=1$, then $n=3$, and since $\frac{n-1}{2}=1$, then our statement holds in this case. Thus we can assume $n \geq 4$ and $|V(Q_{1})| \geq 2$. Hence $(Q_{1},Q_{j})$ is a good pair for $j>1$ unless $Q_{1}$ is too large. If $n$ is odd, then $|V(Q_{1})| = \frac{n-1}{2}$. If $|V(Q_{2})|=1$, $\frac{n-1}{2}+1+1=n$, implying $n=3$. So $|V(Q_{2})| \geq 2$, and $(Q_{2},Q_{3})$ is a good pair unless $|V(Q_{2})|=\frac{n-1}{2}$, in which case $|V(Q_{3})|=1$.

If $n$ is even and $v \in V(Q_{1})$, then $|V(Q_{1})|=\frac{n}{2}-1$. Again, if $|V(Q_{2})|=1$ then $\frac{n}{2}-1+1+1=n$, implying $n=2$. So $|V(Q_{2})| \geq 2$, and $(Q_{2},Q_{3})$ is a good pair unless $|V(Q_{2})|=\frac{n}{2}$, implying $|V(Q_{3})|=1$. (Note here we'd need to relabel the components so they are in descending order of size.)
Finally, if $n$ is even and $v \notin V(Q_{1})$, then $|V(Q_{1})|=\frac{n}{2}$. If $|V(Q_{2})|=1$, then $\frac{n}{2}+1+1=n$, implying $n=4$. However, then $|V(Q_{3})|=1$ and our statement holds. If $n \geq 5$, then $|V(Q_{2})| \geq \frac{n}{2}-1$ else $(Q_{2},Q_{3})$ is a good pair. Since we must have three components, $|V(Q_{2})|=\frac{n}{2}-1$ and $|V(Q_{3})|=1$.
Either way, our components have the desired size.
\end{proof}

\begin{lemma}
\label{lemma:KNkvcc}
Let $G,v,\mathcal{B}$ and $(H,(Q_1, \dots, Q_p))$ be as in Lemma~\ref{lemma:KNknorareconfig}. If $v \notin Q_{3}$, then the vertex in $Q_{3}$ is in a different colour class to $v$.
\end{lemma}
\begin{proof}
By Lemma~\ref{lemma:KNkcompbd}, $|V(Q_{3})|=1$. Let $x$ be the vertex in $Q_{3}$. Assume for the sake of contradiction that $x,v$ are in colour class $X_{s}$. If $n$ is odd then $v \in V(Q_{1})$ or $V(Q_{2})$, but these components have the same order, by Lemma~\ref{lemma:KNkcompbd}. If $n$ is even, $v \in V(Q_{2})$, since otherwise $v$ is in a component of order $\frac{n}{2}$, again by Lemma~\ref{lemma:KNkcompbd}. So without loss of generality, $v \in V(Q_{2})$. Define the hitting set $H'$ as follows: from $H$, add the edges from $v$ to $Q_{2}$, and then remove the edges from $x$ to $Q_{2}$. Since $xv \notin E(G)$, the components of $G-H'$ are $Q_{1}, (Q_{2}-\{v\}) \cup \{x\}$ and $\{v\}$ (as $x,v$ are in the same colour class, $(Q_{2}-\{v\}) \cup \{x\}$ is connected). The orders of the components have not changed, and $v$ has not been placed into a component of order $\frac{n}{2}$, so this is a hitting set by Lemma~\ref{lemma:KNlegithit}. $|H'| = |H| + (|V(Q_{2})|-|V(Q_{2}) \cap X_{s}|) - (|V(Q_{2})|-(|V(Q_{2}) \cap X_{s}|))=|H|$. Since $v$ is now in a component of higher index, this contradicts condition \textbf{(p+1)}.
\end{proof}

The previous lemmas give a good idea of the structure of the components of $G-H$. When dealing with the complete graph, this was sufficient. However, in the case of the complete multipartite graph, we also need to know how the components of $G-H$ interact with the colour classes of $G$. As we might expect, in the optimal case, each colour class is essentially split evenly across the two large components $Q_1$ and $Q_2$. In order to show this, however, we need to be careful about the parity of $n$ and the parities of $n_1, \dots, n_k$. Recall that we label the colour classes $X_1, \dots, X_n$. For the following section, we assume that $G$ is a complete multipartite graph such that $k \geq 2$ and $G$ is not a star, and as such we have only three components by Lemma~\ref{lemma:KNkthreecomp}.

\begin{definition}
Let $X^{*}_{i} := X_{i} \cap (V(Q_{1}) \cup V(Q_{2}))$, and say $X^{*}_{i}$ is \emph{even} or \emph{odd} depending on the parity of its order.
\end{definition}

\begin{definition}
\begin{itemize*}
\item A colour class $X_{i}$ is called \emph{balanced} if $|V(Q_{1}) \cap X_{i}| = |V(Q_{2}) \cap X_{i}|$.
\item A colour class $X_{i}$ is \emph{$Q_{1}$-skew} if $|V(Q_{1}) \cap X_{i}| \geq |V(Q_{2}) \cap X_{i}|+1$. When $|V(Q_{1}) \cap X_{i}| = |V(Q_{2}) \cap X_{i}|+1$, we say $X_{i}$ is \emph{just-$Q_{1}$-skew}.
\item A colour class $X_{i}$ is \emph{$Q_{2}$-skew} if $|V(Q_{1}) \cap X_{i}|+1 \leq |V(Q_{2}) \cap X_{i}|$. When $|V(Q_{1}) \cap X_{i}|+1 = |V(Q_{2}) \cap X_{i}|$, we say $X_{i}$ is \emph{just-$Q_{2}$-skew}.
\item $(X_{i},X{j})$ is called a \emph{skew pair} if $X_{i}$ is $Q_{1}$-skew and $X_{j}$ is $Q_{2}$-skew.
\end{itemize*}
For simplicity, if $X_{i}$ is $Q_{1}$-skew or $Q_{2}$-skew, then we say $X_{i}$ is \emph{skew}. Similarly if $X_{i}$ is just-$Q_{1}$-skew or just-$Q_{2}$-skew, then we say $X_{i}$ is \emph{just-skew}.
\end{definition}

We say $G$ is an \emph{exception} if $n$ is even, and there is a colour class $X_{s}$ such that $|V(Q_{1}) \cap X_{s}| = |V(Q_{1})|-1$ and $|V(Q_{2}) \cap X_{s}| = |V(Q_{2})|-1$.

\begin{lemma}
\label{lemma:KNkbalanceskew}
Let $G$ be a complete multipartite graph $G:=K_{n_{1}, \dots, n_{k}}$ such that $k \geq 2$, $n>4,k$ and $G$ is neither a star nor an exception, $v$ a vertex of $G$ chosen from a largest colour class, $\mathcal{B}$ a canonical line-bramble for $v$, and $(H, (Q_1,Q_2, Q_3))$ a good labelling. If $(X_{i},X{j})$ is a skew pair, then both $X_{i}$ and $X_{j}$ are just-skew.
\end{lemma}
\begin{proof}
Since no colour class can be both $Q_{1}$-skew and $Q_{2}$-skew, $i \neq j$. Since $n \geq 5$, by Lemma~\ref{lemma:KNkcompbd}, both $Q_{1}$ and $Q_{2}$ contain at least two vertices, and thus intersect at least two colour classes.

First, we show that both $X^{*}_{i}$ and $X^{*}_{j}$ contain a vertex other than $v$. If $X^{*}_{i} = \emptyset$, then $X_{i}$ is not skew. So now assume $X^{*}_{i} \neq \emptyset$. Similarly, $X^{*}_{j} \neq \emptyset$. If $X^{*}_{i} = \{v\}$, then by Lemma~\ref{lemma:KNkvcc}, $X_{i} \cap V(Q_{3}) = \emptyset$, and so $X_{i} = \{v\}$. But since $v$ is in a largest colour class, every colour class has order one, and as such $k=n$, which contradicts one of our assumptions on $n$. Thus both $X^{*}_{i}$ and $X^{*}_{j}$ contain a vertex other than $v$, and since $X_{i}$ is $Q_{1}$-skew and $X_{j}$ is $Q_{2}$-skew, there are vertices $x \in (V(Q_{1}) \cap X_{i}) - \{v\}$ and $y \in (V(Q_{2}) \cap X_{j}) - \{v\}$. Then define the hitting set $H'$ as follows: remove the edges from $x$ to $V(Q_{2})$ from $H$, add the edges from $x$ to $V(Q_{1})-X_i$, remove the edges from $y$ to $V(Q_{1})-\{x\}$, and add the edges from $y$ to $V(Q_{2}) \cup \{x\}$. Now $G-H'$ has components $(Q_{1} - \{x\}) \cup \{y\}, (Q_{2} - \{y\}) \cup \{x\}$ and $Q_{3}$, assuming that $(Q_{1} - \{x\}) \cup \{y\}$ and $(Q_{2} - \{y\}) \cup \{x\}$ are in fact connected (which we now prove). 

If $(Q_{1} - \{x\}) \cup \{y\}$ is not connected, then it intersects only one colour class, which must be $X_{j}$ as $y \in X_{j}$. Since $x \in X_{i}$, it follows that $|V(Q_{1}) \cap X_{j}| = |V(Q_{1})|-1$. Since $X_{j}$ is $Q_{2}$-skew, $$|V(Q_{1})| = |V(Q_{1}) \cap X_{j}|+1 \leq |V(Q_{2}) \cap X_{j}| \leq |V(Q_{2})|.$$ Since $|V(Q_{1})| \geq |V(Q_{2})|$, we have $|V(Q_{1})|=|V(Q_{2})|$, and each inequality in the above equation is an equality. In particular, $|V(Q_{2}) \cap X_{j}|=|V(Q_{2})|$, and thus $V(Q_{2}) \subseteq X_{j}$. But $Q_{2}$ intersects at least two colour classes, which is a contradiction. Thus $(Q_{1} - \{x\}) \cup \{y\}$ is a connected component of $G-H'$.

If $(Q_{2} - \{y\}) \cup \{x\}$ is not connected, then it intersects only one colour class, which must be $X_{i}$ as $x \in X_{i}$. Since $y \in X_{j}$, it follows that $|V(Q_{2}) \cap X_{i}| = |V(Q_{2})|-1$. Since $X_{i}$ is $Q_{1}$-skew, $$|V(Q_{1})| \geq |V(Q_{1}) \cap X_{i}| \geq |V(Q_{2}) \cap X_{i}|+1 = |V(Q_{2})|.$$ By Lemma~\ref{lemma:KNkcompbd}, either $|V(Q_{1})| = |V(Q_{2})|$ (when $n$ is odd) or $|V(Q_{1})| = |V(Q_{2})|+1$ (when $n$ is even). If $|V(Q_{1}) \cap X_{i}| = |V(Q_{1})|$, then $V(Q_{1}) \subseteq X_{i}$, contradicting our result that $Q_{1}$ intersects at least two colour classes. Otherwise $|V(Q_{1}) \cap X_{i}| = |V(Q_{1})|-1$, which can only happen when $n$ is even. In this case, since $|V(Q_{1}) \cap X_{i}| = |V(Q_{1})|-1$ and $|V(Q_{2}) \cap X_{i}| = |V(Q_{2})|-1$, $G$ is an exception. This contradiction shows that $(Q_{2} - \{y\}) \cup \{x\}$ is a connected component of $G-H'$.

Thus $G-H'$ has components $(Q_{1} - \{x\}) \cup \{y\}, (Q_{2} - \{y\}) \cup \{x\}$ and $Q_{3}$. Hence the orders of the components have not changed. As the vertex $v$ has not changed components, $H'$ is a legitimate hitting set. But since $H$ is the minimum hitting set by condition \textbf{(0)}, $|H'| \geq |H|$. Hence 
\begin{align*}
|H'| = &|H| - (|V(Q_{2})|-|V(Q_{2}) \cap X_{i}|) + (|V(Q_{1})|-|V(Q_{1}) \cap X_{i}|) \\
&- (|V(Q_{1})|-1-|V(Q_{1}) \cap X_{j}|) + (|V(Q_{2})|+1-|V(Q_{2}) \cap X_{j}|) \\
\geq &|H|.
\end{align*}
Which implies $$|V(Q_{2}) \cap X_{i}| + |V(Q_{1}) \cap X_{j}| \geq |V(Q_{1}) \cap X_{i}| + |V(Q_{2}) \cap X_{j}| -2.$$ Since $X_{i}$ is $Q_{1}$-skew and $X_{j}$ is $Q_{2}$-skew, $$|V(Q_{1}) \cap X_{i}| + |V(Q_{2}) \cap X_{j}| -2 \geq |V(Q_{2}) \cap X_{i}| + |V(Q_{1}) \cap X_{j}| \geq |V(Q_{1}) \cap X_{i}| + |V(Q_{2}) \cap X_{j}| -2.$$ This only holds if every inequality is actually an equality. That is, $X_{i}$ is just-$Q_{1}$-skew and $X_{j}$ is just-$Q_{2}$-skew.
\end{proof}

\begin{lemma}
\label{lemma:KNkskewisjust}
Let $G$ be a complete multipartite graph $G:=K_{n_{1}, \dots, n_{k}}$ such that $k \geq 2$, $n>k$ and $G$ is not a star, $v$ a vertex of $G$ chosen from a largest colour class, $\mathcal{B}$ a canonical line-bramble for $v$, and $(H, (Q_1,Q_2, Q_3))$ a good labelling. If $X_{i}$ is skew, then $X_{i}$ is just-skew.
\end{lemma}
\begin{proof}
Suppose $G$ is not an exception and $n > 4$. If there exists a $Q_{1}$-skew colour class $X_{s}$ and a $Q_{2}$-skew colour class $X_{t}$, then either $(X_{s},X_{i})$ or $(X_{i},X_{t})$ is a skew pair, and by Lemma~\ref{lemma:KNkbalanceskew}, $X_{i}$ is just-skew, as required.

Alternatively, either no colour class is $Q_{1}$-skew or no colour class is $Q_{2}$-skew. Suppose, for the sake of contradiction, there is a skew colour class $X_{j}$ that is not just-skew. In the first case, for all $\ell$, $|V(Q_{1}) \cap X_{\ell}| \leq |V(Q_{2}) \cap X_{\ell}|$, and $|V(Q_{1}) \cap X_{j}|+2 \leq |V(Q_{2}) \cap X_{j}|$. Thus $$|V(Q_{1})|+2 = (\!\!\!\!\!\!\sum_{1 \leq \ell \leq k, \ell \neq j}\!\!\!\!\!\!\!\!|V(Q_{1}) \cap X_{\ell}|) + |V(Q_{1}) \cap X_{j}|+2 \leq (\!\!\!\!\!\!\sum_{1 \leq \ell \leq k, \ell \neq j}\!\!\!\!\!\!\!\!|V(Q_{2}) \cap X_{\ell}|) + |V(Q_{2}) \cap X_{j}| = |V(Q_{2})|.$$ This contradicts $|V(Q_{1})| \geq |V(Q_{2})|$. Similarly, in the second case, $|V(Q_{1})| \geq |V(Q_{2})|+2$, which contradicts Lemma~\ref{lemma:KNkcompbd}. Thus if $n \geq 5$ and $G$ is not an exception, then our statement holds.

Consider the case when $G$ is an exception. Then $|V(Q_{1}) \cap X_{s}| = |V(Q_{1})|-1$ and $|V(Q_{2}) \cap X_{s}| = |V(Q_{2})|-1$. Since $n$ is even, by Lemma~\ref{lemma:KNkcompbd}, $|V(Q_{1})| = |V(Q_{2})|+1$, so $X_{s}$ is just-skew. There are exactly two other vertices of $Q_{1} \cup Q_{2}$, one in each component, which we label $x$ and $y$ respectively. If $x$ and $y$ are in the same colour class, then that colour class is balanced. Otherwise, $x$ and $y$ are in different colour classes, each of which intersects $Q_{1} \cup Q_{2}$ in one vertex. Such a colour class is just-skew, as required.

Finally, consider the case $n \leq 4$. Then $|V(Q_{1}) \cup V(Q_{2})| \leq 3$. Thus either $|V(Q_{1})|=|V(Q_{2})|=1$, or $|V(Q_{1})|=2$ and $|V(Q_{2})|=1$. If $X_{i}$ is not just-skew, then $X_{i}$ contains at least two vertices in some component. Thus, the only possibility to consider is when $|V(Q_{1}) \cap X_{i}| = 2$. But then $Q_{1}$ is not connected, since both vertices are in the same colour class, which contradicts the fact that $Q_1$ is a connected component.

Thus $X_{i}$ is just-skew.
\end{proof}

From Lemma~\ref{lemma:KNkskewisjust} and Lemma~\ref{lemma:KNkcompbd}, we get the following results about $|Q_{1} \cap X_{i}|$ and $|Q_{2} \cap X_{i}|$:
\begin{corollary}
\label{corollary:KNksizerules}
Let $G,v,\mathcal{B}$ and $(H, (Q_1,Q_2, Q_3))$ be as in Lemma~\ref{lemma:KNkskewisjust}.
If a colour class $X_{i}$ does not intersect $Q_{3}$, then 
\begin{itemize*}
\item if $X_{i}$ is balanced, then $|Q_{1} \cap X_{i}| = |Q_{2} \cap X_{i}| = \frac{n_{i}}{2}$
\item if $X_{i}$ is $Q_{1}$-skew, then $|Q_{1} \cap X_{i}| = \frac{n_{i}+1}{2}$ and $|Q_{2} \cap X_{i}| = \frac{n_{i}-1}{2}$
\item if $X_{i}$ is $Q_{2}$-skew, then $|Q_{1} \cap X_{i}| = \frac{n_{i}-1}{2}$ and $|Q_{2} \cap X_{i}| = \frac{n_{i}+1}{2}$
\end{itemize*}
\end{corollary}

\begin{corollary}
\label{corollary:KNksizerules2}
Let $G,v,\mathcal{B}$ and $(H, (Q_1,Q_2, Q_3))$ be as in Lemma~\ref{lemma:KNkskewisjust}.
If a colour class $X_{i}$ does intersect $Q_{3}$, then $|V(Q_{3}) \cap X_{i}|=1$ and
\begin{itemize*}
\item if $X_{i}$ is balanced, then $|Q_{1} \cap X_{i}| = |Q_{2} \cap X_{i}| = \frac{n_{i}-1}{2}$
\item if $X_{i}$ is $Q_{1}$-skew, then $|Q_{1} \cap X_{i}| = \frac{n_{i}}{2}$ and $|Q_{2} \cap X_{i}| = \frac{n_{i}-2}{2}$
\item if $X_{i}$ is $Q_{2}$-skew, then $|Q_{1} \cap X_{i}| = \frac{n_{i}-2}{2}$ and $|Q_{2} \cap X_{i}| = \frac{n_{i}}{2}$
\end{itemize*}
\end{corollary}

\begin{lemma}
\label{lemma:KNknoofskew}
Let $G,v,\mathcal{B}$ and $(H, (Q_1,Q_2, Q_3))$ be as in Lemma~\ref{lemma:KNkskewisjust}. If $n$ is odd, then there is an equal number of $Q_{1}$-skew and $Q_{2}$-skew colour classes. If $n$ is even, then there is one more $Q_{1}$-skew colour class than there are $Q_{2}$-skew colour classes.
\end{lemma}
\begin{proof}
Say there are $a$ $Q_{1}$-skew colour classes and $b$ $Q_{2}$-skew colour classes. By Lemma~\ref{lemma:KNkskewisjust}, if $X_{i}$ is $Q_{1}$-skew, then $|V(Q_{1}) \cap X_{i}| = |V(Q_{2}) \cap X_{i}|+1$, and if $X_{i}$ is $Q_{2}$-skew, then $|V(Q_{1}) \cap X_{i}| = |V(Q_{2}) \cap X_{i}|-1$. Thus $$|V(Q_{1})|=\sum_{1 \leq i \leq k} |V(Q_{1}) \cap X_{i}| = (\sum_{1 \leq i \leq k} |V(Q_{2}) \cap X_{i}|) + a - b = |V(Q_{2})| +a-b.$$ If $n$ is odd, then by Lemma~\ref{lemma:KNkcompbd}, $|V(Q_{1})|=|V(Q_{2})|$, so $a=b$, as required. When $n$ is even, $|V(Q_{1})|=|V(Q_{2})|+1$, so $a=b+1$.
\end{proof}

From Lemma~\ref{lemma:KNkthreecomp}, Lemma~\ref{lemma:KNkcompbd}, Corollary~\ref{corollary:KNksizerules} and Corollary~\ref{corollary:KNksizerules2}, we get the following result that summarises this section:
\begin{theorem}
\label{theorem:Hfornonreg}
Let $G$ be a complete multipartite graph $G:=K_{n_{1}, \dots, n_{k}}$ such that $k \geq 2, n>k$ and $G$ is not a star, $v$ a vertex of $G$ chosen from a largest colour class, $\mathcal{B}$ a canonical line-bramble for $v$, and $(H, (Q_1, \dots, Q_p))$ a good labelling. Then $p=3$. If $n$ is odd, then $|V(Q_{1})|=|V(Q_{2})|=\frac{n-1}{2}$ and $|V(Q_{3})|=1$, and if $n$ is even, then $|V(Q_{1})|=\frac{n}{2}$, $|V(Q_{2})|=\frac{n}{2}-1$ and $|V(Q_{3})|=1$. For a colour class $X_{i}$, $$\ceil{\frac{n_{i}-2}{2}} \leq |V(Q_{1}) \cap X_{i}|,|V(Q_{2}) \cap X_{i}| \leq \floor{\frac{n_{i}+1}{2}}.$$
\end{theorem}

Now we can use Theorem~\ref{theorem:Hfornonreg} to determine a lower bound on $\tw(L(G))$.
\begin{theorem}
\label{theorem:lbonH}
Let $G$ be a complete multipartite graph $G:=K_{n_{1}, \dots, n_{k}}$ where $k \geq 2$. Then $\tw(L(G)) + 1 = \bn(L(G)) \geq \frac{1}{2}\left(\sum\limits_{1 \leq i < j \leq k} n_{i}n_{j}\right) + \frac{3}{4}k^2 -kn  - \frac{3}{4}k + n$.
\end{theorem}
\begin{proof}
First, consider the case when $k \geq 2, n>k$ and $G$ is not a star. Then choose some vertex $v$ in a largest colour class of $G$, a canonical line-bramble for $v$ denoted $\mathcal{B}$ and a good labelling $(H, (Q_1, \dots, Q_p))$. It is sufficient to determine a lower bound on $|H|$, since $H$ is a minimum hitting set for $\mathcal{B}$ by condition \textbf{(0)}, and since $\mathcal{B}$ forces the existence of a bramble of $L(G)$ of the same order by Lemma~\ref{lemma:KNlbisb}. Using Theorem~\ref{theorem:Hfornonreg}, we can determine the structure of $H$. The set $H$ contains all edges with an endpoint in $Q_1$ and an endpoint in $Q_2$; simply count these edges. By Theorem~\ref{theorem:Hfornonreg}, $|V(Q_{1}) \cap X_{i}|,|V(Q_{2}) \cap X_{i}| \geq \ceil{\frac{n_{i}}{2} - 1}$. As $n_i,n_j \geq 1$, it follows that $|V(Q_{1}) \cap X_{i}||V(Q_{2}) \cap X_{j}| \geq \left(\frac{n_{i}}{2}-1\right)\left(\frac{n_{j}}{2}-1\right) - \frac{1}{4}$. So we count the edges from $Q_1$ to $Q_2$ as follows:
\begin{align*}
\sum_{i \neq j} |V(Q_{1}) \cap X_{i}||V(Q_{2}) \cap X_{j}| \geq &\sum_{i \neq j} \left(\frac{n_{i}}{2}-1\right)\left(\frac{n_{j}}{2}-1\right) - \frac{1}{4} \\
= &\frac{1}{4}\left(\sum_{i \neq j} n_{i}n_{j}\right) - (k-1)n + \frac{3}{4}k(k-1) \\
= &\frac{1}{2}\left(\sum_{1 \leq i < j \leq k}n_{i}n_{j}\right) + \frac{3}{4}k^2 -kn - \frac{3}{4}k + n.
\end{align*}
This gives the required lower bound on $|H|$ in this case.

It remains to check the cases when either $n=k$ or $G$ is a star. When $n=k$, $G$ is simply the complete graph, and our lower bound follows by Theorem~\ref{theorem:KN}. If $G$ is a star, then $L(G)$ is a complete graph, and the lower bound follows by inspection.
\end{proof}

Using the same techniques as in the above proof, we can also determine an upper bound on $|H|$. We do this now. Note when considering the upper bound, we also need to account for the edges from $Q_3$ into the components $Q_1,Q_2$, but there are not many of these edges.
\begin{corollary}
\label{corollary:ubonH}
Let $G,v,\mathcal{B}$ and $(H,(Q_1,\dots,Q_p))$ be as in Theorem~\ref{theorem:Hfornonreg}. Then \newline $|H| \leq \frac{1}{2}\left(\sum_{1 \leq i < j \leq k}n_{i}n_{j}\right) +\frac{1}{2}n(k+1) + \frac{1}{4}k(k - 1) -1.$
\end{corollary}

Also, our results in this section give a more detailed understanding of $H$ when $G$ is regular.
\begin{theorem}
\label{theorem:Hforreg}
Let $G$ be a complete regular $k$-partite graph $G:=K_{c, \dots, c}$, such that $k \geq 2, n>k$, $v$ a vertex of $G$ chosen from a largest colour class, $\mathcal{B}$ a canonical line-bramble for $v$, and $(H, (Q_1, \dots, Q_p))$ a good labelling. Then $p=3$. If $n$ is odd, then $|V(Q_{1})|=|V(Q_{2})|=\frac{n-1}{2}$ and $|V(Q_{3})|=1$ and
\begin{itemize*}
\item for one colour class $X_{i}$, we have $|V(Q_{1}) \cap X_{i}|=|V(Q_{2}) \cap X_{i}| = \frac{c-1}{2}$ and $|V(Q_{3}) \cap X_{i}|=1$,
\item for $\frac{k-1}{2}$ other colour classes $X_{i}$, we have $|V(Q_{1}) \cap X_{i}|=\frac{c+1}{2}$ and $|V(Q_{2}) \cap X_{i}| = \frac{c-1}{2}$,
\item for the remaining $\frac{k-1}{2}$ colour classes $X_{i}$, we have $|V(Q_{1}) \cap X_{i}|=\frac{c-1}{2}$ and $|V(Q_{2}) \cap X_{i}| = \frac{c+1}{2}$.
\end{itemize*}
If $n$ is even, then $|V(Q_{1})|=\frac{n}{2}$, $|V(Q_{2})|=\frac{n}{2}-1$ and $|V(Q_{3})|=1$. If $n$ is even and $c$ is odd, then 
\begin{itemize*}
\item for one colour class $X_{i}$, we have $|V(Q_{1}) \cap X_{i}|=|V(Q_{2}) \cap X_{i}| = \frac{c-1}{2}$ and $|V(Q_{3}) \cap X_{i}|=1$,
\item for $\frac{k}{2}$ other colour classes $X_{i}$, we have $|V(Q_{1}) \cap X_{i}|=\frac{c+1}{2}$ and $|V(Q_{2}) \cap X_{i}| = \frac{c-1}{2}$,
\item for the remaining $\frac{k}{2}-1$ colour classes $X_{i}$, we have $|V(Q_{1}) \cap X_{i}|=\frac{c-1}{2}$ and $|V(Q_{2}) \cap X_{i}| = \frac{c+1}{2}$.
\end{itemize*}
Finally, if $n$ is even and $c$ is even, then
\begin{itemize*}
\item for one colour class $X_{i}$, we have $|V(Q_{1}) \cap X_{i}|=\frac{c}{2}$, $|V(Q_{2}) \cap X_{i}| = \frac{c}{2}-1$ and $|V(Q_{3}) \cap X_{i}|=1$,
\item for the other $k-1$ colour classes $X_{i}$, we have $|V(Q_{1}) \cap X_{i}|=|V(Q_{1}) \cap X_{i}|=\frac{c}{2}$.
\end{itemize*}
\end{theorem}
\begin{proof}
Since $G$ is regular and $n > k \geq 2$, $G$ is not a star. The statements about the number and order of the components of $G-H$ all follow from Lemma~\ref{lemma:KNkthreecomp} and Lemma~\ref{lemma:KNkcompbd}. Since $n=ck$, when $n$ is odd, $c$ is odd and $k$ is odd. When $n$ is even, at least one of $c$ and $k$ are even. Then from Corollary~\ref{corollary:KNksizerules}, Corollary~\ref{corollary:KNksizerules2} and Lemma~\ref{lemma:KNknoofskew}, the rest of the theorem follows.
\end{proof}

\section{Path decompositions of the Complete Multipartite Graph}
\label{section:KNknonregtree}
We reuse the following notation from the previous section: $G$ is a complete multipartite graph $K_{n_{1}, \dots, n_{k}}$, that is not an independent set (that is, $k \geq 2$), complete graph (that is, $n > k$) or a star. (Recall we have already proven Theorem~\ref{theorem:KNk} for such graphs.) The vertex $v$ of $G$ is chosen from a largest colour class, and $\mathcal{B}$ a canonical line-bramble for $v$. $(H, (Q_1, \dots, Q_p))$ is a good labelling, and by Theorem~\ref{theorem:Hfornonreg} we can assume $p=3$. $X_{1}, \dots, X_{k}$ are the colour classes of $G$ such that $|X_{i}|=n_{i}$.

From the results of the previous section, it is possible to determine the order of a minimum hitting set $H$. However, first we find a path decomposition of $L(G)$ with width expressed in terms of $H$, as this will make things easier. 

Now we define path decomposition for $L(G)$ as follows. Let $T$ be the underlying path. Since $T$ is a path, it makes sense to refer to a bag \emph{left} or \emph{right} of another bag, depending on the relative positions of the corresponding nodes in $T$. If a bag is to the right of another bag and the nodes which index them are adjacent in $T$, then we say it is \emph{directly right}. Similarly define \emph{directly left}. For a vertex $u$ of $G$, let $\deg_{i}(u)$ be the number of edges in $G$ incident to $u$ with the other endpoint in the component $Q_{i}$.

First, label the vertices of $Q_{1}$ by $x_{1}, \dots, x_{|V(Q_{1})|}$ in some order, which we will specify later. Similarly, label the vertices of $Q_{2}$ by $y_{1}, \dots, y_{|V(Q_{2})|}$, again in an order we will later specify. Finally, by Theorem~\ref{theorem:Hfornonreg}, $Q_{3}$ contains a single vertex, which we label $z$.

Then define the following bags:
\begin{itemize*}
\item $\gamma := H = \{uw \in E(G): u,w$ are in different components of $G-H \}$,
\item for for $1 \leq i \leq |V(Q_{1})|$, $$\!\!\!\!\!\!\!\!\!\!\!\!\!\!\!\!\alpha_{i} := \{ x_{\ell}u \in E(G) : u \in V(Q_{1}), 1 \leq \ell \leq i\} \cup  \{x_{j}w \in E(G) : w \in V(G) - V(Q_1), i \leq j \leq |V(Q_{1})|\},$$
\item for $1 \leq i \leq |V(Q_{2})|$, $$\!\!\!\!\!\!\!\!\!\!\!\!\!\!\!\!\beta_{i} := \{ y_{\ell}u \in E(G) : u \in V(Q_{2}), 1 \leq \ell \leq i\} \cup \{y_{j}w \in E(G) : w \in V(G)-V(Q_{2}), i \leq j \leq |V(Q_{2})|\}.$$
\end{itemize*}

Each bag is indexed by a node of $T$. Left-to-right, the nodes of $T$ index the bags in the following order: $\beta_{|V(Q_{2})|}, \dots, \beta_{1}, \gamma, \alpha_{1}, \dots, \alpha_{|V(Q_{1})|}$. Let $\mathcal{X}$ denote the collection of bags.
We claim this defines a path decomposition $(T,\mathcal{X})$ for $L(G)$, independent of our ordering of $Q_{1}$ and $Q_{2}$.

\begin{lemma}
\label{lemma:Tistd}
Let $G$ be a complete multipartite graph $G:=K_{n_{1}, \dots, n_{k}}$ such that $k \geq 2$, $n>k$ and $G$ is not a star, $v$ a vertex of $G$ chosen from a largest colour class, $\mathcal{B}$ a canonical line-bramble for $v$, and $(H, (Q_1,Q_2, Q_3))$ a good labelling.
Then $(T,\mathcal{X})$ is a path decomposition of $L(G)$, irrespective of the ordering used on $Q_1$ and $Q_2$.
\end{lemma}
\begin{proof}
Consider $uw \in E(G)$. We require that the nodes indexing the bags containing $uw$ induce a non-empty connected subpath of $T$. Firstly, assume that $u$ and $w$ are in different components of $G-H$. If $u=x_{i}$ and $w=y_{j}$, then $uw \in \beta_{j}, \dots, \beta_{1}, \gamma, \alpha_{1}, \dots, \alpha_{i}$, meaning $uw$ is in precisely this sequence of bags. If $u=x_{i}$ and $w=z$, then $uw \in \gamma, \alpha_{1}, \dots, \alpha_{i}$. If $u=y_{j}$ and $w=z$, then $uw \in \beta_{j}, \dots, \beta_{1}, \gamma$.

Secondly, assume that $u$ and $w$ are in the same component of $G-H$, which is either $Q_{1}$ or $Q_{2}$, since by Theorem~\ref{theorem:Hfornonreg}, $|V(Q_{3})|=1$. If $u,w \in V(Q_{1})$, then let $u=x_{i}$ be the vertex of smaller label. Then $uw \in \alpha_{i}, \dots, \alpha_{|V(Q_{1})|}$. If $u,w \in V(Q_{2})$, then similarly let $u=y_{i}$ be the vertex of smaller label. Then $uw \in \beta_{|V(Q_{2})|}, \dots, \beta_{i}$. This shows that the nodes indexing the bags containing $uw$ induce a non-empty connected subpath of $T$.

All that remains is to show that if two edges are incident at a vertex in $G$ (that is, the edges are adjacent in $L(G)$), then there is a bag of $\mathcal{X}$ containing both of them. Now if the shared vertex of the two edges is $x_{i} \in V(Q_{1})$, then by inspection both edges are in $\alpha_{i}$. If the shared vertex is $y_{j} \in V(Q_{2})$, then both edges are in $\beta_{j}$. Finally, if the shared vertex is $z$, then both edges are in $\gamma$.
\end{proof}

Now we determine the width of $(T,\mathcal{X})$, which is one less than the order of the largest bag. To do so, we use a specific labelling of $Q_{1} \cup Q_{2}$. We do this in two different ways, depending on whether $G$ is regular.

In our first ordering, label the vertices $x_{1}, \dots, x_{|V(Q_{1})|}$ in order of non-decreasing size of the colour class containing $x_{i}$, and do the same for $y_{1}, \dots, y_{|V(Q_{2})|}$. We denote this ordering as the \emph{red ordering}.

\begin{lemma}
\label{lemma:KNknralpha}
Let $G,v,\mathcal{B},(H,(Q_1,Q_2,Q_3))$ and $(T,\mathcal{X})$ be as in Lemma~\ref{lemma:Tistd}, but assume the ordering on $Q_1$ and $Q_2$ is the red ordering. Then $|\alpha_{i}| \leq |\alpha_{1}| + n-2$, for all $1 \leq i \leq |V(Q_{1})|$. 
\end{lemma}
\begin{proof}
We will show that $|\alpha_{i}| \leq |\alpha_{i-1}|+2$ for all $i$. This implies that $|\alpha_{i}| \leq |\alpha_{1}| + 2(i-1)$. Since $i \leq |V(Q_{1})|$ and $|V(Q_{1})| \leq \frac{n}{2}$ by Lemma~\ref{lemma:KNlegithit}, this is sufficient.
\begin{align*}
\alpha_{i} = &\{ x_{\ell}u, x_{j}w \in E(G) : u \in V(Q_{1}), w \in V(G)-V(Q_{1}), 1 \leq \ell \leq i, i \leq j \leq |V(Q_{1})|\} \\
= &\{x_{\ell}u \in E(G) : u \in V(Q_{1}), 1 \leq \ell \leq i\} \cup \{x_{j}w \in E(G) : w \in V(G)-V(Q_{1}), i \leq j \leq |V(Q_{1})|\}.
\end{align*} This is a disjoint union. Let $X_{s},X_{t}$ be the colour classes such that $x_{i-1} \in X_{s}$ and $x_{i} \in X_{t}$, and note that it is possible $s=t$. Then 
\begin{align*}
|\alpha_{i}| - |\alpha_{i-1}| = &|\{x_{\ell}u \in E(G) : u \in V(Q_{1}), 1 \leq \ell \leq i\}| \\
&- |\{x_{\ell}u \in E(G) : u \in V(Q_{1}), 1 \leq \ell \leq i-1\}| \\
&+ |\{x_{j}w \in E(G) : w \in V(G)-V(Q_{1}), i \leq j \leq |V(Q_{1})|\}| \\
&- |\{x_{j}w \in E(G) : w \in V(G)-V(Q_{1}), i-1 \leq j \leq |V(Q_{1})|\}| \\
\leq &\deg_{1}(x_{i}) -|\{x_{i-1}w \in E(G): w \in V(G)-V(Q_{1})\}| \\
= &\deg_{1}(x_{i})-(\deg_{G}(x_{i-1})-\deg_{1}(x_{i-1})) \\
= &\deg_{1}(x_{i})-(n - n_{s} - \deg_{1}(x_{i-1})) \\
= &|V(Q_{1})|-|V(Q_{1} \cap X_{t})|-(n - n_{s} - |V(Q_{1})|+|V(Q_{1} \cap X_{s})|) \\
= &2|V(Q_{1})|+n_{s}-|V(Q_{1} \cap X_{t})|-n-|V(Q_{1} \cap X_{s})|.
\end{align*}

Assume for the sake of contradiction that $|\alpha_{i}| - |\alpha_{i-1}| > 2$. Then:
$$2|V(Q_{1})| + n_{s} > n  + |V(Q_{1}) \cap X_{s}| + |V(Q_{1}) \cap X_{t}| +2.$$
By the ordering of the vertices in $Q_{1}$, $n_{t} \geq n_{s}$. Then by Theorem~\ref{theorem:Hfornonreg}, $$|V(Q_{1}) \cap X_{s}| + |V(Q_{1}) \cap X_{t}| \geq \frac{n_{s}-2}{2} + \frac{n_{t}-2}{2} \geq n_{s} - 2.$$ Hence $2|V(Q_{1})| + n_{s} > n+n_{s}-2+2$; that is, $2|V(Q_{1})| > n$. But $|V(Q_{1})|>\frac{n}{2}$ contradicts Lemma~\ref{lemma:KNlegithit}.
\end{proof}

By symmetry we have:
\begin{lemma}
\label{lemma:KNknrbeta}
Let $G,v,\mathcal{B},(H,(Q_1,Q_2,Q_3))$ and $(T,\mathcal{X})$ be as in Lemma~\ref{lemma:Tistd}, but assume the ordering on $Q_1$ and $Q_2$ is the red ordering. Then$|\beta_{i}| \leq |\beta_{1}| + n-2$, for all $1 \leq i \leq |V(Q_{2})|$. 
\end{lemma}

\begin{lemma}
\label{lemma:KNknrmaxbag}
Let $G,v,\mathcal{B},(H,(Q_1,Q_2,Q_3))$ and $(T,\mathcal{X})$ be as in Lemma~\ref{lemma:Tistd}. The maximum bag size of $(T,\mathcal{X})$, using the red ordering, is at most $|H| + 2n-2$.
\end{lemma}
\begin{proof}
By Lemma~\ref{lemma:KNknralpha} and Lemma~\ref{lemma:KNknrbeta}, the maximum size of a bag right of $\gamma$ is at most $|\alpha_{1}|+n-2$, and left of $\gamma$ it is $|\beta_{1}|+n-2$. By inspection, the edges in $\alpha_{1}-\gamma$ are all adjacent to $x_{1}$. Hence there are at most $n$ of them. Thus $|\alpha_{1}| \leq |\gamma|+n$. Similarly $|\beta_{1}| \leq |\gamma|+n$. Since $\gamma=H$, this is sufficient.
\end{proof}

Given this, we can determine an upper bound on $\pw(L(G))$.

\begin{theorem}
\label{theorem:KNktwlb}
Let $G$ be a complete multipartite graph $G:=K_{n_{1}, \dots, n_{k}}$ where $k \geq 2$. Then $\pw(G) \leq \frac{1}{2}\left(\sum_{1 \leq i < j \leq k}n_{i}n_{j}\right) + \frac{1}{4}k^2 + \frac{1}{2}kn - \frac{1}{4}k + \frac{5}{2}n -4$
\end{theorem}
\begin{proof}
If $G,v,\mathcal{B},(H,(Q_1,Q_2,Q_3))$ and $(T,\mathcal{X})$ are as in Lemma~\ref{lemma:Tistd}, then we have a path decomposition of width at most $|H| + 2n - 3$ by Lemma~\ref{lemma:KNknrmaxbag}. (Note as $k \geq 2$, it follows $n \geq 2$ and so $2n-3$ is positive.) Then our result follows from Corollary~\ref{corollary:ubonH}. In the remaining cases, $G$ is either a complete graph or a star, and this result follows by Theorem~\ref{theorem:KN} or inspection, respectively.
\end{proof}

Thus Theorem~\ref{theorem:KNk} follows from Theorem~\ref{theorem:lbonH} and Theorem~\ref{theorem:KNktwlb}.
%

When $G$ is regular, that is, $n_{1} = \dots = n_{k}$, we can get a more accurate bound on the treewidth and pathwidth. Define $c:=n_{1}$ to be the size of each colour class. We need a different ordering of the vertices $x_{1}, \dots, x_{|Q_{1}|}$ and $y_{1}, \dots, y_{|Q_{2}|}$ to obtain our result. In order to do this, we recall the notion of a skew colour class, as defined in Section~\ref{section:KNk}, and the associated results. 
First consider a colour class $X_{i}$ that does not intersect $Q_{3}$. If $X_{i}$ is balanced, then say every vertex of $X_{i}$ is \emph{Type 1}. If $X_{i}$ is $Q_{1}$-skew, then each vertex in $Q_{1} \cap X_{i}$ is \emph{Type 1} and each vertex in $Q_{2} \cap X_{i}$ is \emph{Type 2}. If $X_{i}$ is $Q_{2}$-skew, then each vertex in $Q_{1} \cap X_{i}$ is \emph{Type 2} and each vertex in $Q_{2} \cap X_{i}$ is \emph{Type 1}. Finally, each vertex in the remaining colour class (that does intersect $Q_{3}$) is \emph{Type 3}. Thus each vertex of $V(G)-z$ is either Type 1, 2 or 3.
Label the vertices of $Q_{1}$ in order $x_{1}, \dots, x_{|V(Q_{1})|}$ by first labelling Type 1 vertices, then Type 2 vertices, and finally Type 3 vertices. Do the same for $y_{1}, \dots, y_{|V(Q_{2})|}$. We denote this ordering as the \emph{blue ordering}.

\begin{lemma}
\label{lemma:RTDgoodtype}
Let $G$ be a complete $k$-partite graph with $n > k$, $v$ a vertex of $G$, $\mathcal{B}$ a canonical line-bramble for $v$ and $(H,(Q_1,Q_2,Q_3))$ a good labelling. If $k \geq 3$, then $Q_{1}$ contains at least two Type 1 vertices, and $Q_{2}$ contains at least one Type 1 vertex. If $k=2$ and $c \geq 3$, then $Q_{1}$ contains at least two Type 1 vertices, and $Q_{2}$ contains at least one Type 1 or Type 2 vertex.
\end{lemma}
\begin{proof}
If $X_{i}$ is a colour class that does not intersect $Q_{3}$, then it intersects both of $Q_{1}$ and $Q_{2}$---if not, then by Lemma~\ref{lemma:KNkskewisjust}, $|X_{i}|=1$ and $G$ is the complete graph. Since we are trying to find Type 1 and Type 2 vertices, from now on we only consider colour classes that do not intersect $Q_{3}$. If $k \geq 5$, then there are at least four colour classes that do not intersect $Q_{3}$. From Theorem~\ref{theorem:Hforreg}, there are either at least two $Q_{1}$-skew and $Q_{2}$-skew colour classes, or at least four balanced colour classes. Even if each such colour class intersects each of $Q_{1}$ and $Q_{2}$ only once, there are still enough colour classes of the correct skew to get all our required Type 1 vertices. Similarly, if $k=4$ and $c$ is odd, then there are two $Q_{1}$-skew colour classes and one $Q_{2}$-skew colour class, and if $k=4$ and $c$ is even, there are three balanced colour classes. This is again sufficient.

If $k=3$, then by Theorem~\ref{theorem:Hforreg} again, there are enough $Q_{2}$-skew or balanced colour classes to ensure that $Q_{2}$ has at least one Type 1 vertex. However, if $n$ is odd, there is only one $Q_{1}$-skew colour class. In this case, $c$ is odd, and so $c \geq 3$. Thus that colour class contains at least two vertices in $Q_{1}$. Thus $Q_{1}$ has two Type 1 vertices.

Now assume $k=2$ and $c \geq 3$. If $c$ is odd, there is one $Q_{1}$-skew colour class, again by Theorem~\ref{theorem:Hforreg}. This colour class contains at least two vertices in $Q_{1}$ and one in $Q_{2}$, which satisfies our requirement, now that $Q_{2}$ only requires a Type 2 vertex. If $c$ is even, then there is one balanced colour class. $c \geq 3$, so as it is even, $c \geq 4$ and each component contains two vertices from this colour class. This is sufficient.
\end{proof}

The following lemma strengthens Lemma~\ref{lemma:KNknralpha} for the case when $G$ is regular.

\begin{lemma}
\label{lemma:KNkregalpha}
Let $G$ be a complete $k$-partite graph with $n > k$, $v$ a vertex of $G$, $\mathcal{B}$ a canonical line-bramble for $v$ and $(H,(Q_1,Q_2,Q_3))$ a good labelling. Let $(T,\mathcal{X})$ be our path decomposition where $Q_1$ and $Q_2$ are ordered by the blue ordering.
If $k \geq 3$ or $c \geq 3$, then $|\alpha_{1}| \geq |\alpha_{2}| \geq \dots \geq |\alpha_{|V(Q_{1})|}|$.
\end{lemma}
\begin{proof}
We will show that $|\alpha_{i}| \leq |\alpha_{i-1}|$ for all $i$. We can write $\alpha_{i}$ as the disjoint union
$$\alpha_{i} = \{x_{\ell}u \in E(G) : u \in V(Q_{1}), 1 \leq \ell \leq i\} \cup \{x_{j}w \in E(G) : w \in V(G)-V(Q_{1}), i \leq j \leq |V(Q_{1})|\}.$$
Let $X_{s},X_{t}$ be the colour classes such that $x_{i-1} \in X_{s}$ and $x_{i} \in X_{t}$, and note that it is possible that $s=t$. Define $r:=|\{x_{i}x_{f} \in E(G) : f < i\}|$. Then 
\begin{align*}
|\alpha_{i}| - |\alpha_{i-1}| = &|\{x_{\ell}u \in E(G) : u \in V(Q_{1}), 1 \leq \ell \leq i\}| \\
&- |\{x_{\ell}u \in E(G) : u \in V(Q_{1}), 1 \leq \ell \leq i-1\}| \\
&+ |\{x_{j}w \in E(G) : w \in V(G)-V(Q_{1}), i \leq j \leq |V(Q_{1})|\}| \\
&- |\{x_{j}w \in E(G) : w \in V(G)-V(Q_{1}), i-1 \leq j \leq |V(Q_{1})|\}| \\
= &\deg_{1}(x_{i}) - r -|\{x_{i-1}w \in E(G)| w \in V(G)-V(Q_{1})\}| \\
= &\deg_{1}(x_{i}) - r -(\deg_{G}(x_{i-1})-\deg_{1}(x_{i-1})) \\
= &\deg_{1}(x_{i}) - r -(n - n_{s} - \deg_{1}(x_{i-1})) \\
= &|V(Q_{1})|-|V(Q_{1} \cap X_{t})|-r-(n - n_{s} - |V(Q_{1})|+|V(Q_{1} \cap X_{s})|) \\
= &2|V(Q_{1})|+n_{s}-r-|V(Q_{1} \cap X_{t})|-n-|V(Q_{1} \cap X_{s})|.
\end{align*}

Assume for the sake of contradiction that $|\alpha_{i}| - |\alpha_{i-1}| > 0$. Then:
$$2|V(Q_{1})| + n_{s} > n  + r + |V(Q_{1}) \cap X_{s}| + |V(Q_{1}) \cap X_{t}|.$$
There are two cases to consider. Firstly, say that both $x_{i-1}$ and $x_{i}$ are Type 1. So $X_{s}$ and $X_{t}$ are both balanced or $Q_{1}$-skew, and neither intersects $Q_{3}$. Since $G$ is regular, $n_{t}=n_{s}$. Then by Corollary~\ref{corollary:KNksizerules}, $|V(Q_{1}) \cap X_{s}| + |V(Q_{1}) \cap X_{t}| \geq \frac{n_{s}}{2} + \frac{n_{t}}{2} = n_{s}$. Hence $2|V(Q_{1})| + n_{s} > n + n_{s} +r \geq n+n_{s}$, so $2|V(Q_{1})| > n$, which contradicts Lemma~\ref{lemma:KNlegithit}.

Secondly, since we ordered our vertices by non-decreasing type, we can assume $x_{i}$ does not have Type 1. However, by Lemma~\ref{lemma:RTDgoodtype}, $Q_{1}$ has at least two Type 1 vertices, $x_{a}$ and $x_{b}$. Note if two vertices of $Q_{1}$ are in the same colour class, they have the same type, so we know that $x_{a}$ and $x_{b}$ are in a different colour class to $x_{i}$. Also, $a,b < i$, thus $r \geq 2$. Since $n_{t}=n_{s}$, by Theorem~\ref{theorem:Hfornonreg}, $|V(Q_{1}) \cap X_{s}| + |V(Q_{1}) \cap X_{t}| \geq \frac{n_{s}-2}{2} + \frac{n_{t}-2}{2} = n_{s}-2$. Hence $2|V(Q_{1})| + n_{s} > n + n_{s}-2 +r \geq n+n_{s}$, so $2|V(Q_{1})| > n$, which again contradicts Lemma~\ref{lemma:KNlegithit}.
\end{proof}

We must also consider the equivalent argument for bags to the left of $\gamma$, as we did in the general case. However, here the arguments are not quite the same.

\begin{lemma}
\label{lemma:KNkregbeta}
Let $G,v,\mathcal{B},(H,(Q_1,Q_2,Q_3))$ and $(T,\mathcal{X})$ be as in Lemma~\ref{lemma:KNkregalpha}. If $k \geq 3$ or $c \geq 3$, then $|\beta_{1}| \geq |\beta_{2}| \geq \dots \geq |\beta_{|V(Q_{2})|}|$.
\end{lemma}
\begin{proof}
We will show that $|\beta_{i}| \leq |\beta_{i-1}|$ for all $i$. We can write $\beta_{i}$ as the disjoint union
$$
\beta_{i} = \{y_{\ell}u \in E(G) : u \in V(Q_{2}), 1 \leq \ell \leq i\} \cup \{y_{j}w \in E(G) : w \in V(G)-V(Q_{2}), i \leq j \leq |V(Q_{2})|\}.
$$ Let $X_{s},X_{t}$ be the colour classes such that $y_{i-1} \in X_{s}$ and $y_{i} \in X_{t}$, and note that it is possible that $s=t$. Define $r:=|\{y_{i}y_{f} \in E(G) : f < i\}|$. Then 
\begin{align*}
|\beta_{i}| - |\beta_{i-1}| = &|\{y_{\ell}u \in E(G) : u \in V(Q_{2}), 1 \leq \ell \leq i\}| \\
&- |\{y_{\ell}u \in E(G) : u \in V(Q_{2}), 1 \leq \ell \leq i-1\}| \\
&+ |\{y_{j}w \in E(G) : w \in V(G)-V(Q_{2}), i \leq j \leq |V(Q_{2})|\}| \\
&- |\{y_{j}w \in E(G) : w \in V(G)-V(Q_{2}), i-1 \leq j \leq |V(Q_{2})|\}| \\
= &\deg_{2}(y_{i}) - r -|\{y_{i-1}w \in E(G)| w \in V(G)-V(Q_{2})\}| \\
= &\deg_{2}(y_{i}) - r -(\deg_{G}(y_{i-1})-\deg_{2}(y_{i-1})) \\
= &\deg_{2}(y_{i}) - r -(n - n_{s} - \deg_{2}(y_{i-1})) \\
= &|V(Q_{2})|-|V(Q_{2} \cap X_{t})|-r-(n - n_{s} - |V(Q_{2})|+|V(Q_{2} \cap X_{s})|) \\
= &2|V(Q_{2})|+n_{s}-r-|V(Q_{2} \cap X_{t})|-n-|V(Q_{2} \cap X_{s})|.
\end{align*}

Assume for the sake of contradiction that $|\beta_{i}| - |\beta_{i-1}| > 0$. Then:
$$2|V(Q_{2})| + n_{s} > n  + r + |V(Q_{2}) \cap X_{s}| + |V(Q_{2}) \cap X_{t}|.$$
There are two cases to consider. Firstly, say that neither of $y_{i}$ and $y_{i-1}$ have Type 3. So neither $X_{s}$ nor $X_{t}$ intersects $Q_{3}$. $G$ is regular, so $n_{t}=n_{s}$. By Corollary~\ref{corollary:KNksizerules}, $|V(Q_{2}) \cap X_{s}| + |V(Q_{2}) \cap X_{t}| \geq \frac{n_{s}-1}{2} + \frac{n_{t}-1}{2} = n_{s}-1$. Hence $2|V(Q_{2})| + n_{s} > n+r+n_{s}-1 \geq n+n_{s}-1$, and so $2|V(Q_{2})| > n-1$. However, Theorem~\ref{theorem:Hforreg} states that $|V(Q_{2})| \leq \frac{n-1}{2}$, so this is a contradiction.

Secondly, $y_{i}$ has Type 3. By Lemma~\ref{lemma:RTDgoodtype}, $Q_{2}$ contains at least one non-Type 3 vertex; this will be of a different colour class to $y_{i}$ and have a lower numbered index. Hence $r \geq 1$. By Theorem~\ref{theorem:Hfornonreg}, $|V(Q_{2}) \cap X_{s}| + |V(Q_{2}) \cap X_{t}| \geq \frac{n_{s}-2}{2} + \frac{n_{t}-2}{2} = n_{s}-2$, and hence $2|V(Q_{2})| + n_{s} > n+r+n_{s}-2 \geq n+n_{s}-1$. Again, this contradictions Theorem~\ref{theorem:Hforreg}.
\end{proof}

\begin{lemma}
\label{lemma:KNkgamma}
Let $G,v,\mathcal{B},(H,(Q_1,Q_2,Q_3))$ and $(T,\mathcal{X})$ be as in Lemma~\ref{lemma:KNkregalpha}. If $k \geq 3$ or $c \geq 3$, then $|\alpha_{1}| \leq |\gamma|$ and $|\beta_{1}| \leq |\gamma|$.
\end{lemma}
\begin{proof}
By inspection, $\alpha_{1} = \{x_{1}u, uw \in E(G) : u \in V(Q_{1}), w \in V(G)-V(Q_{1})\}$. Thus the edges of the form $x_{1}u$ are the only edges in $\alpha_{1}$ not in $\gamma$, and the edges between $Q_{2}$ and $Q_{3}$ (all of which are adjacent to $z$) are the only edges in $\gamma$ not in $\alpha_{1}$. Thus $|\alpha_{1}| + \deg_{2}(z) - \deg_{1}(x_{1}) = |\gamma|$. Suppose for the sake of contradiction that $|\alpha_{1}| > |\gamma|$. Say $x_{1} \in X_{s}$ and $z \in X_{t}$. By Lemma~\ref{lemma:RTDgoodtype}, $x_{1}$ has Type 1, so $s \neq t$. Substituting $\deg_{2}(z) = |V(Q_{2})| - |V(Q_{2}) \cap X_{t}|$ and $\deg_{1}(x_{1}) = |V(Q_{1})| - |V(Q_{1}) \cap X_{s}|$ gives
$$|V(Q_{1})| - |V(Q_{2})| > |V(Q_{1}) \cap X_{s}| - |V(Q_{2}) \cap X_{t}|.$$
By Theorem~\ref{theorem:Hforreg}, $|V(Q_{1})|-|V(Q_{2})| \leq 1$. Similarly, since $X_{t}$ intersects $Q_{3}$, $|V(Q_{2}) \cap X_{t}| = \frac{c-1}{2}$ if $c$ is odd, and $|V(Q_{2}) \cap X_{t}|=\frac{c-2}{2}$ if $c$ is even. Since $X_{s} \cap Q_{3} = \emptyset$ and $x_{1}$ has Type 1, $|V(Q_{1}) \cap X_{s}| \geq \frac{c}{2}$. Hence $|V(Q_{1}) \cap X_{s}| - |V(Q_{2}) \cap X_{t}| \geq \frac{1}{2}$ if $c$ is odd, or $1$ if $c$ is even. However, this value is an integer, so $|V(Q_{1}) \cap X_{s}| - |V(Q_{2}) \cap X_{t}| \geq 1$, implying $|V(Q_{1})|-|V(Q_{2})| > 1$,  which is a contradiction of Theorem~\ref{theorem:Hforreg}.

Now we consider $\beta_{1} =\{y_{1}u, uw \in E(G) : u \in V(Q_{2}), w \in V(G)-V(Q_{2})\}$. Suppose for the sake of contradiction that $|\beta_{1}| > |\gamma|$. Let $y_{1} \in X_{s}$ and $z \in X_{t}$. By Lemma~\ref{lemma:RTDgoodtype}, $x_{1}$ has Type 1 or Type 2, so $s \neq t$. Performing substitutions as we did in the $\alpha_{1}$ case gives
$$|V(Q_{2})| - |V(Q_{1})| > |V(Q_{2}) \cap X_{s}| - |V(Q_{1}) \cap X_{t}|.$$
Since $X_{s}$ does not intersect $Q_{3}$ and $X_{t}$ does, by Theorem~\ref{theorem:Hforreg}, $|V(Q_{2}) \cap X_{s}| \geq \frac{c-1}{2}$ and $|V(Q_{1}) \cap X_{t}| = \frac{c-1}{2}$ or $\frac{c}{2}$. Thus $|V(Q_{2}) \cap X_{s}| - |V(Q_{1}) \cap X_{t}| \geq 0$ or $-\frac{1}{2}$, but since it is an integer, $|V(Q_{2}) \cap X_{s}| - |V(Q_{1}) \cap X_{t}| \geq 0$, implying $|V(Q_{2})|-|V(Q_{1})|>0$, which contradicts Theorem~\ref{theorem:Hforreg}.
\end{proof}

By Lemmas~\ref{lemma:KNkregalpha}, \ref{lemma:KNkregbeta} and \ref{lemma:KNkgamma}, $\gamma$ is the largest bag in all but a few cases. Recall $\gamma=H$. Hence, together with Theorem~\ref{theorem:Hforreg}, we get the following result.

If $G$ is a regular $k$-partite graph such that $n>k$, and either $k \geq 3$ or $c \geq 3$, then $$\pw(L(G))=\tw(L(G)) = |H|-1.$$

We now accurately determine $|H|$ when $G$ is regular.

We can determine $|H|$ by calculating the number of edges between $Q_{1}$ and $Q_{2}$, and the number of edges adjacent to $z \in Q_{3}$. Theorem~\ref{theorem:Hforreg} gives us all we require. It follows that:
$$
|H| = 
\begin{cases}
\frac{c^{2}k^{2}}{4} - \frac{c^{2}k}{4} + \frac{ck}{2} - \frac{c}{2} + \frac{k}{4} - \frac{1}{4} &\text{, if $ck$ odd} \\
\frac{c^{2}k^{2}}{4} - \frac{c^{2}k}{4} + \frac{ck}{2} - \frac{c}{2} &\text{, if $c$ even} \\
\frac{c^{2}k^{2}}{4} - \frac{c^{2}k}{4} + \frac{ck}{2} - \frac{c}{2} + \frac{k}{4} - \frac{1}{2} &\text{, if $k$ even, $c$ odd} 
\end{cases}
$$

This gives the exact answer for the treewidth and pathwidth of the line graph of the $(n-c)$-regular complete $k$-partite graph, when $n>k$ and either $k \geq 3$ or $c \geq 3$. When $n=k$, we determine the treewidth using by Theorem~\ref{theorem:KN}. When $k=2$ and $c=2$, $G$ is a 4-cycle and $\pw(K_{2,2}) = \tw(K_{2,2})=2$, which satisfies our result by inspection. This proves Theorem~\ref{theorem:KNkreg}. 

\bibliographystyle{amsplain}
\bibliography{twbib}

\end{document}